\documentclass[12pt]{amsart}

\usepackage{geometry}
\usepackage{amsmath} 
\usepackage{amssymb} 
\usepackage{amsthm}
\usepackage{amsfonts}            
\usepackage{mathrsfs}
\usepackage[all]{xy}
\usepackage{mathtools}
\usepackage{mathabx}
\usepackage{colonequals}
\usepackage{tikz} % picture
\usepackage{microtype}

\usepackage{cite}
\usepackage{comment} %% comment package

\usepackage[pagebackref]{hyperref} %% \hypersetup
\hypersetup{  %% Set up hyperref color
	colorlinks=true,
	linkcolor=red!50!black,
	citecolor=green,
	urlcolor=magenta,
}
\usepackage[lite,abbrev,msc-links,alphabetic]{amsrefs}

\usepackage{setspace}

\newcommand{\GL}{\mathrm{GL}}
\newcommand{\SL}{\mathrm{SL}}

\newcommand{\U}{\mathrm{U}}
\newcommand{\GU}{\mathrm{GU}}

\newcommand{\Lie}{\mathrm{Lie}}
\newcommand{\FJ}{\mathrm{FJ}}
\newcommand{\End}{\mathrm{End}}
\newcommand{\Hom}{\mathrm{Hom}}
\newcommand{\Aut}{\mathrm{Aut}}

\newcommand{\Gal}{\mathrm{Gal}}
\newcommand{\RSZ}{\mathrm{RSZ}}
\newcommand{\Spec}{\mathrm{Spec}}
\newcommand{\Spf}{\mathrm{Spf}}
\newcommand{\Sh}{\mathrm{Sh}}
\newcommand{\Res}{\mathrm{Res}}

\newcommand{\vol}{\mathrm{vol}}
\newcommand{\Hk}{\mathrm{Hk}}
\newcommand{\BBP}{\mathrm{BBP}}
\newcommand{\BB}{\mathrm{BB}}
\newcommand{\codim}{\mathrm{codim}}
\newcommand{\Int}{\mathrm{Int}}
\newcommand{\Orb}{\mathrm{Orb}}
\newcommand{\Split}{\mathrm{Split}}
\newcommand{\Inert}{\mathrm{Inert}}

\newcommand{\rs}{\mathrm{rs}}
\newcommand{\Ch}{\mathrm{Ch}}
\newcommand{\id}{\mathrm{id}}
\newcommand{\wh}{\widehat}
\newcommand{\val}{\mathrm{val}}
\newcommand{\Ker}{\mathrm{Ker}}
\newcommand{\Nm}{\mathrm{Nm}}
\newcommand{\As}{\mathrm{As}}

\newcommand{\Herm}{\mathrm{Herm}}
\newcommand{\can}{\mathrm{can}}
\newcommand{\charpol}{\mathrm{charpol}}
\newcommand{\del}{\operatorname{\partial Orb}}

\newcommand{\BA}{\mathbb A}
\newcommand{\BC}{\mathbb C}

\newcommand{\BR}{\mathbb R}

\newcommand{\BJ}{\mathbb J}
\newcommand{\BZ}{\mathbb Z}
\newcommand{\BQ}{\mathbb Q}
\newcommand{\BG}{\mathbb G}
\newcommand{\BF}{\mathbb F}
\newcommand{\BE}{\mathbb E}
\newcommand{\BL}{\mathbb L}
\newcommand{\BV}{\mathbb V}
\newcommand{\BX}{\mathbb X}

\newcommand{\CH}{\mathrm{CH}}

\newcommand{\CA}{\mathcal A}
\newcommand{\CM}{\mathcal M}
\newcommand{\CN}{\mathcal N}
\newcommand{\CE}{\mathcal E}
\newcommand{\CF}{\mathcal F}

\newcommand{\CC}{\mathcal C}
\newcommand{\CG}{\mathcal G}
\newcommand{\CO}{\mathcal O}
\newcommand{\CP}{\mathcal P}
\newcommand{\CX}{\mathcal X}
\newcommand{\CY}{\mathcal Y}
\newcommand{\CZ}{\mathcal Z}
\newcommand{\CalS}{\mathcal{S}}

\newcommand{\CS}{\mathcal{S}}

\newcommand{\ov}{\overline}
\newcommand{\wt}{\widetilde}

\theoremstyle{definition}        
\newtheorem{definition}{Definition}[subsection] 

\newtheorem{remark}[definition]{Remark}

\theoremstyle{plain}
\newtheorem{theorem}[definition]{Theorem}
         
\newtheorem{proposition}[definition]{Proposition}
\newtheorem{corollary}[definition]{Corollary}
\newtheorem{conjecture}[definition]{Conjecture}

\begin{document}

\title[Mirabolic cycles and twisted AFL]{Non-reductive special cycles and twisted arithmetic fundamental lemma}
\author{Zhiyu Zhang}
\address{Department of Mathematics, Stanford University,  450 Jane Stanford Way Building 380, Stanford, CA 94305}
\email{zyuzhang@stanford.edu}

\date{}
\setcounter{tocdepth}{1}
\setcounter{secnumdepth}{3}

\maketitle

\begin{abstract}
We consider arithmetic analogs of the relative Langlands program and applications of new non-reductive geometry. Firstly, we introduce mirabolic special cycles, which produce special cycles to many Hodge type Rapoport-Zink spaces via pullbacks e.g. Kudla--Rapoport cycles. Secondly, we formulate arithmetic intersection problems for these cycles and formulate a method of arithmetic induction. As a main example, we formulate arithmetic twisted Gan--Gross--Prasad conjectures on unitary Shimura varieties and prove a key twisted arithmetic fundamental lemma using mirabolic special cycles, arithmetic inductions and Weil type relative trace formulas. 
\end{abstract}
\pagenumbering{roman}
\tableofcontents

\pagenumbering{arabic}

\section{Introduction}

\subsection{Arithmetic relative Langlands program and non-reductive $p$-adic geometry}

Non-reductive groups are used a lot in the theory of modular forms and automorphic forms for reductive groups, see e.g. Fourier coefficients, parabolic inductions, Jacobi groups and Jacobi forms. In the relative Langlands program (see e.g. recent work of Ben-Zvi--Sakellaridis--Venkatesh \cite{BZSV}), we have a lot of interesting examples of period integrals and $L$-functions for a reductive group $G$ acting on a nice space $X$. In pursuit of an arithmetic relative Langlands program relating algebraic cycles and $L$-functions (see e.g. the survey \cite{Zhang2024survey}), we would like to have geometric objects similar to Shimura varieties attached to these $G$-spaces. A test case will be $X=G/H$ for a subgroup $H \leq G$. However, non-reductive $H$ may not even give complex algebraic varieties. A typical example is the embedding (related to Fourier coefficients for modular forms and horocycle flows for dynamic systems)
$$\BG_a=\begin{pmatrix}
1 & * \\ 0 & 1    
\end{pmatrix} \to \SL_2
$$
which on adelic quotients gives real circles $S^1$ on complex modular curves $X_0(N)$.

One motivation of the paper is new non-reductive geometry in the nonarchimedean
world (or how to find $p$-adic analogs of circles). Via mirabolic special cycles and pullbacks, in this paper we construct local geometric cycles for \emph{a nice representation} $X=V$ of $G$.

We consider the following local arithmetic relative Langlands program. Let $p$ be a prime and $F_0/\BQ_p$ be a local field. Let $(H, \mu_H, b_H) \to (G, \mu_G, b_G)$ be an embedding of local Shimura data over $F_0$. For compatible level $K_H \leq H(F_0), K_G \leq G(F_0)$, we expect an embedding of relevant Rapoport--Zink spaces (which exists in the parahoric Hodge type cases \cite{HP17}\cite{kim2018}\cite{HK2019Hodgeparahoric})
\[
\CN_H:=\CN_{H, \mu_H, b_H, K_H} \to \CN_G:=\CN_{G, \mu_G, b_G, K_G}
\]
as formal schemes over $p$-adic integers, compatible with actions of $J_{b_H}(F_0) \to J_{b_G}(F_0)$. If $b_G$ is basic, then $J_{b_G}$ is an inner form of $G$. For $g \in J_{b_G}(F_0)/J_{b_H}(F_0)$, we have the special cycle 
 \[
\CZ_{G/H}(g):=g \CN_H \to \CN_G.
\]

Let $X$ be a $G$-variety with a companion $J_{b_G}$-variety $X_{b_G}$ over $F_0$. In good cases, for regular $x \in X_{b_G}(F_0)$ we expect the existence of \emph{special Rapoport--Zink cycles} (with level $K_G$)
\[
\CZ_X(x) \to \CN_G
\]
such that $g \CZ_X(x)= \CZ_X(gx), \, \forall g \in J_{b_G}(F_0)$. When $X=G/H$ is affine spherical, we recover the cycle $\CZ_{H,G}(g)=g \CN_H$ as above. In the group case $G=H \times H$ and $X=H$, we recover the graph of any automorphism $g \in J_{b_H}$ on $\CN_H$. In this paper, we will also encounter the Galois twisted group case and study twisted graphs. We expect such constructions for a good class of $G$-spaces (see e.g. \cite{BZSV}) and admits a derived enhancement to general $x \in X(F_0)$. 

In this paper, we focus on the case $X=V$ is a nice representation of $G$. Our first main result is a discovery of new algebraic cycles on general linear Rapoport--Zink spaces --- the mirabolic special cycles. Then, we may achieve the above construction via pullbacks from universal mirabolic special cycles attached to $V$. For instance, we recover Kudla--Rapoport type cycles on basic unitary Rapoport--Zink spaces \cite{KR2011local} for $X_{b_G}=\BV$ where $J_{b_G}=\U(\BV)$. A more general theory of Rapoport--Zink cycles will be studied in a future work \cite{ZZZ}. 

Let $F/F_0$ be the unramified quadratic extension. Let $V$ be a $n$-dimensional $F$-vector space. Choosing a framing object $(\BX, \iota_\BX)$, consider the basic general linear Rapoport--Zink space for $\Res_{F/F_0} \GL(V)$:
\[
\CN^\GL_n \to \Spf O_{\breve{F_0}}
\]
as the moduli space of $(X, \iota, \rho)$ over test schemes $S$, where $X$ is a strict formal $O_{F_0}$-module over $S$ of relative $O_{F_0}$-height $2n$ and dimension $n$, $\iota$ is a $O_F$-action on $X$ of Kottwitz signature $(n-1,1)$, $\rho$ is a $O_F$-linear height $0$ quasi-isogeny $\rho: X \times S/p \to \BX \times S/p$. Let $\BV$ be the $n$-dimensional $F$-vector space of special quasi-homomorphisms attached to $(\BX, \iota_\BX)$ as in the work of Kudla-Rapoport \cite{KR2011local}. For non-zero $u \in \BV$, the work of Kudla--Rapoport \cite{KR2011local} introduces Kudla--Rapoport divisors $\CZ(u)$ on the unitary Rapoport zink space $\CN_n$, which can be thought as an arithmetic geometrization of spherical Weil representations.

\begin{theorem}[Theorem \ref{prop: geometry of mirabolic cycle}]
For non-zero $u \in \BV$, we have mirabolic special cycles 
$$\CZ^\GL(u) \to \CN^\GL_n
$$ which are relative Cartier divisors. For a natural embedding of the unitary Rapoport--Zink space $\CN_n \to \CN_n^\GL$, we recover the Kudla--Rapoport divisors  via the pullback diagram:
     \[
\xymatrix{
	\CZ(u)  \ar[r] \ar[d] &  \CZ^\GL(u)  \ar[d]  \\  
	\CN_n  \ar[r] & \CN^\GL_n }
.\]
\end{theorem}

See Theorem \ref{prop: geometry of mirabolic cycle} for more geometric properties, which is similar to the work of Kudla--Rapoport \cite[Theorem 1.1]{KR2011local}. Let $\GL(u) \leq \GL(\BV)$ be the mirabolic subgroup defined as the stabilizer of $u$, then the generic fiber of $\CZ^\GL(u)$ is an analog of local Shimura varieties for the (non-reductive) mirabolic group $\GL(u)$. Given any Rapoport--Zink space $\CN_G$ and a $n$-dimensional representation $V$ of $G$ inducing an embedding $\CN_G \to \CN_n^\GL$, we obtain \emph{special linear cycles} $Z_V(u) \to \CN_G$ ($u \not=0$) as (derived) pull backs of $Z^\GL(u)$, see Remark \ref{defn: pullback mirabolic speical cycles}. These cycles  are arithmetic analogs of fibers for the variety of triples in the theory of Coulomb branches \cite{BFN}, which may be explored further. From moduli definitions of mirabolic special cycles, we have 
\[
\CZ_{V_1 \oplus V_2}(u_1,u_2)=\CZ_{V_1}(u_1) \cap_{\CN_G} \CZ_{V_2}(u_2), \, \, \forall  u_i \not =0.
\]
In this way, Kudla--Rapoport conjectures \cite{KR2011local}\cite{LiZhangKR2019} can be viewed as studying arithmetic degrees of (a derived version of) special linear cycles attached to the representation $V^{\oplus \dim V}$ of $\U(V)$. Therefore, we expect these special linear cycles will be helpful to formulate and study more arithmetic intersection problems, which will be studied in a future work. For instance, see \cite{WZhangBesselAFL2022} and the role of Kudla-Rapoport cycles in arithmetic fundamental lemmas \cite{AFL-Wei2019}\cite{AFL-JEMS}\cite{ZZhang2021} and Kudla--Rapoport conjectures \cite{LiZhangKR2019}. In this paper, they also play a key role in the proof of a twisted arithmetic fundamental lemma whose statement does not involve such cycles!

\subsection{Arithmetic intersections and arithmetic induction systems}

Now, we consider arithmetic intersection problems from the above program. Assume that $\CN_G$ is regular. Let $X_1, ..., X_m$ be varieties with $G$-actions such that we have constructed above special cycles $\CZ_{X_i}(x_i) \to \CN_G$ on an open $J_{b_G}$-stable subset $x_i \in U_i \subseteq X_{i,b_G}$. Assume that there exists a dense open $J_{b_G}$-stable subset $U_{X_1, \hdots, X_m} \subseteq \prod_{i=1}^m U_{X_i}$, such that for any $(x_i) \in U_{X_1, \hdots, X_m}$, the intersection $\cap_{i=1}^m \CZ_{X_i}(x_i)$ is a proper scheme, and $\sum_{i=1}^{m} \codim \CZ_{X_i}(x_i)= \dim \CN_G$. Then we consider the $J_{b_G}$-invariant arithmetic invariant functional
\[
\Int_{X_1, \hdots, X_m}^{G}: U_{X_1, \hdots, X_m} \to \BQ,
\]
\[
\Int_{X_1, \hdots, X_m}^{G}(x_1, \hdots, x_m) := \CZ_{X_1}(x_1) \cap^{\BL} \hdots \cap^{\BL} \CZ_{X_m}(x_m) := \chi(\CN_G, \CO_{\CZ_{X_1}(x_1)} \otimes^{\BL} \hdots \otimes^{\BL} \CO_{\CZ_{X_m}(x_m)} ).
\]
Moreover, we expect \emph{arithmetic fundamental lemmas} (AFL) and \emph{arithmetic transfers} (AT) relating $\Int_{X_1, \hdots, X_m}^{G}$ with analytic integrals of certain test functions on these orbits. In good cases these special cycles may be globalized and used to formulate Gross--Zagier type formulas \cite{GrossZagier} for Shimura varieties, with applications to Beilinson--Bloch--Kato conjectures on $L$-functions and motives.  The unitary arithmetic Gan--Gross--Prasad conjectures \cite{GGP-AGGP}\cite{AFL-Invent} and arithmetic Siegel--Weil formulas in the Kudla program \cite{Kudla-Annals97}\cite{KRY-book}\cite{LiZhangKR2019}\cite{Li-Liu}\cite{LiZhangGSpin} are well-studied examples. See also the work of \cite{Madapusi-derived} on derived special cycles and some extensions of the Kudla program.

\subsubsection{Arithmetic induction systems}
To understand the above arithmetic intersection questions, we introduce the notion of arithmetic induction systems for good tuples $(G, X=X_1, Y=X_2)$ (Definition \ref{defn: arithmetic induction system}), which allows an inductive understanding of arithmetic / analytic functionals. The idea is already crucially used in the recent proof of Kudla--Rapoport conjectures \cite{LiZhangKR2019} and arithmetic fundamental lemmas \cite{AFL-Wei2019}\cite{AFL-JEMS}\cite{ZZhang2021}.

In this paper, guided by the above program, we mainly study arithmetic analogs of the twisted Gan--Gross--Prasad conjectures \cite{twistedGGP}\cite{Wang-TGGP} on twisted Asai L-functions. As a key local conjecture, we study and prove twisted arithmetic fundamental lemmas (Conjecture \ref{Twisted AFL conjecture}) via local arithmetic inductions and global methods.

\begin{theorem}[Proposition \ref{geometric side: arithmetic induction}, Proposition \ref{analytic side: arithmetic induction}]
Both sides of the twisted arithmetic fundamental lemma Conjecture \ref{Twisted AFL conjecture} are arithmetic induction systems.
\end{theorem}

The proof uses above mirabolic special cycles. In the sequel \cite{ZZZ}, we will study more examples of arithmetic intersection questions and prove they are arithmetic induction systems.

\subsection{Twisted arithmetic fundamental lemma and arithmetic Gan-Gross-Prasad conjectures}

Let $F/F_0$ be an unramified quadratic extension. We consider
\begin{itemize}
    \item (arithmetic unitary side) Let $V=(V, h_0)$ be a split $n$-dimensional hermitian space over $F$. The group $G:=\GL(V)$ acts naturally on the set of (non-degenerate) hermitian structures $h \in X=\Herm(V)$ on $V$ with two orbits -- split orbits and non-split orbits (determined by isomorphism classes of $(V, h)$). Let $H_{V, h}:=\U(V, h)$, which acts naturally on $X_{V,h}:=\GL(V)/\U(V, h)$ (resp. $Y_{V, h}=V$). For a non-split hermitian space $\BV=(V, h_1)$ and $H_\BV=\U(\BV)$, we consider unramified unitary Rapoport--Zink spaces $\CN_n$ for $(V, h_0)$ with a natural embedding $\CN_n \to \CN_n^\GL$. On $\CN_n$, we have \emph{twisted fixed cycles} $\CN_n^{\Herm}(g)=g \CN_n \cap^{\BL}_{\CN^\GL_n} \CN_n, \forall g \in \GL(\BV)/\U(\BV)$ (resp. \emph{Kudla--Rapoport divisors} $\CZ(u), \forall 0 \not = u \in \BV$ ).
    \item (analytic general linear side) Let $V_0$ be a $n$-dimensional vector space over $F_0$ and $V_0^*$ be its linear dual. Let $H'=\GL(V_0)$.  Let $X'=\GL(V_0)$ with the conjugacy action of $H'$, and $Y'=V_0':=V_0 \times V_0^*$ with the natural action of $H'$.
\end{itemize}

We have a matching of regular semi-simple orbits 
\[
[H' \backslash (X' \times Y')]_\rs \cong [H_V \backslash (X_V \times Y_V)]_\rs \coprod [H_{\BV} \backslash (X_\BV \times Y_\BV)]_\rs, \, \, (\gamma,u') \leftrightarrow (g, u).
\]
with matching characteristic polynomials $\charpol(\gamma)=\charpol(g)$ and inner products $u_2\gamma^i u_1= (g^i u, u), i \in \BZ$. Choose a self-dual lattice $L$ in $V$. Choose an orthogonal basis $\{e_i\}_{i=1}^n$ of $L$, and let $L_0= \sum_{i=1}^n O_{F_0}e_i$. We may assume that $V_0=L_0 \otimes \BQ$. We consider the lattice $L'=L_0 \times L_0^* \subseteq Y'=V_0'$. Define standard test functions (\ref{standard test function})
\[
\Phi_{L'}'=1_{\GL(L_0)} \times 1_{L'} \in C_c^\infty(X' \times Y'), \, \Phi_L= 1_{\Herm(L)} \times 1_{L} \in C_c^\infty(X_V \times Y_V).
\]

\begin{theorem}[Twisted Fundamental Lemma, Theorem \ref{Twisted FL}] 
For any regular semi-simple orbits $(\gamma, u') \in (\GL_n(V_0) \times V_0')(F_0)_\rs$, we have 
\[
\Orb((\gamma,u'), \Phi'_{L'})= \begin{cases}
   \Orb((g,u), \Phi_L), & \text{if $(\gamma,u')$ matches $(g,u) \in (\GL(V)/\U(V) \times V)(F_0)_\rs$.}\\
   0, & \text{if $(\gamma,u')$ matches $(g,u) \in (\GL(\BV)/\U(\BV) \times \BV)(F_0)_\rs$.}  \\
\end{cases}
\]
\end{theorem}

For $(g,u) \in (\GL(\BV)/\U(\BV) \times \BV)(F_0)_\rs$, consider the arithmetic intersection number (\ref{defn: intersection number})
\[
\Int^{\Herm, \BV}(g,u):=(g \CN_n \cap^{\BL}_{\CN^\GL_n} \CN_n)  \cap^\BL_{\CN_n} \CZ(u) \in \BQ.
\]

\begin{theorem}[Twisted Arithmetic Fundamental Lemma, Conjecture \ref{Twisted AFL conjecture}, Theorem \ref{Twisted AFL theorem}]
For any $(g,u) \in (\GL(\BV)/\U(\BV) \times \BV)(F_0)_\rs$ matching $(\gamma,u') \in (\GL_n \times V_0')(F_0)_\rs$, we have
\[			
 \partial \Orb((\gamma,u'),  \Phi'_{L'}) =- \Int^{\Herm, \BV}(g,u) \log q.
\]
\end{theorem}

Compared to the semi-Lie AFL \cite{AFL-Wei2019} which uses the diagonal $\U(V) \to \U(V) \times \U(V)$, our twisted AFL uses the twisted diagonal $\U(V) \to \GL(V)$ which is quite different.

Such twisted AFL is a crucial ingredient for a relative trace formula approach towards arithmetic analogs of twisted Gan--Gross--Prasad conjectures \cite{twistedGGP} using twisted Fourier-Jacobi cycles following \cite{Liu-FJcycles}, for which we give a formulation at the end of this paper.

\begin{conjecture}[Arithmetic Twisted Gan-Gross-Prasad conjecture, Conjecture \ref{conj: ATGGP}]
Let $F/F_0$ (resp. $E/F_0$) be a CM (resp. totally real) quadratic extension of number fields. Let $\mu$ be a conjugate symplectic automorphic character of $\BA_F^\times/F^\times$ of weight one. Let $\pi$ be a cuspidal automorphic representation of $\U(V_{E_0})(\BA_{E_0})$. Let $\Pi$ be the automorphic representation of $\GL_n(\BA_{E})$ as the base change of $\pi$. Assume that $\Pi_\infty$ is relevant \cite[Definition 1.2]{Liu-FJcycles}. Then the following statements are equivalent.
\begin{enumerate}
    \item The twisted arithmetic functional $\FJ_\epsilon \not =0$.
    \item We have $L'(1/2, \Pi, \As_{E/F} \otimes \mu) \not =0$ and 
   $ \Hom_{\BC[\U(V)(\BA_f)]}(\pi_f \otimes \Omega(\mu, \epsilon), \BC) \not =0.$
\end{enumerate}
\end{conjecture}

Allowing such twisting in arithmetic Gan--Gross--Prasad conjectures will be useful to handle arithmetics of motives (e.g. Asai motives), see for instance works of \cite{liu2016hirzebruch} \cite{loeffler2023euler}\cite{loeffler2020p} related to modular forms and elliptic curves.

The proof of our twisted AFL is based on induction and globalization, using above mirabolic special cycles, arithmetic induction systems and the method of \cite{AFL-Wei2019}. Firstly, we show that both sides of twisted AFL are arithmetic induction systems under new relative Cayley maps on hermitian structures. Secondly, we consider globalizations over a totally real field $F_0$, which are expected to be automorphic distributions for $\SL_{2,F_0}$. We show the adelic difference is automorphic, using the modularity of geometric theta series on unitary Shimura varieties with hyperspecial levels \cite{BHKRY}\cite{AFL-JEMS}. We prove local-global decomposition formulas, which indicate the global sides are correct generalizations of the local sides. By such local-global compatibility and arithmetic induction, we deduce many Fourier coefficients of the adelic difference of two global sides vanish. We deduce the difference is a constant by modularity. Hence the local difference vanishes, and we obtain the twisted AFL identity.

Such globalization can be understood in general via what we call \emph{Weil type relative trace formulas}. In a general setup, let $(H_1, G, H_2)$ be a tuple of algebraic groups over a number field $F_0$ where $H_i \subseteq G$ ($i=1,2$) are subgroups and $G$ is reductive. Consider the kernel function
\[
K_f(x, y)= \sum_{\gamma \in G(F)} f(x^{-1}\gamma y), \quad f \in \CS(G(\BA)), x,y \in G(\BA).
\]
Compared to the usual relative trace formula, we need an algebraic representation $V$ of $H_1$. Then the \emph{Weil type relative trace formula} for $(H_1, G, H_2, V)$ is the distribution (which may be not convergent in general)
\begin{equation}
I(f, \phi)= \int_{[H_1] \times [H_2]} K_f(h_1, h_2) \theta_\phi(h_1) dh_1 dh_2, \quad f \in \CS(G(\BA)), \phi \in \CS(V(\BA))
\end{equation}
where theta series $\theta_\phi(g)=\sum_{u \in V(F_0)} \phi(g^{-1}u)$ gives an automorphic realization of $\phi \in \CS(V(\BA))$. If $f \otimes \phi$ has regular supports at a place of $F_0$, then we have a decomposition into orbital integrals
\[
I(f, \phi)= \sum_{(\gamma,u) \in [H \backslash (G/H \times V)]_\rs(F_0) } \Orb((\gamma, u), f \otimes \phi).
\]
In this paper, we consider similar \emph{arithmetic Weil type relative formulas}
\[
I^1(f, \phi):=(Z_V(\phi), R(f)_*\Sh_H)_{\Sh_G}
\]
where $Z_V(\phi), \phi \in \CS(V(\BA_f))$ are global versions of special linear cycles and $R(f), f \in \CS(G(\BA_f))$ are Hecke translations on the Shimura variety $\Sh_G$, and $(-,-)_{\Sh_G}$ is an Arakelov type arithmetic intersection pairing defined using regular integral models.

\subsection{More applications and further questions}

It is interesting to study the geometry of these new special linear cycles (and relations to \cite{BFN}\cite{BZSV}), which could be used in arithmetic relative Langlands program and may have globalization in good cases. We expect our method to give a uniform way of formulating and proving arithmetic fundamental lemmas for other good tuples $(G, X, Y)$ combining the idea of twisted fixed cycles and linear special cycles. The same method can give an arithmetic transfer result for vertex lattices as in \cite{ZZhang2021}. We could formulate similar questions for $H=\U(V) \to G=\U(V_{E_0})$ where $E_0/F_0$ is any quadratic \'etale algebra. For simplicity of notations, we choose not to include arithmetic transfers and general twisted cases.

Finally, we give an overview of the paper. Firstly we introduce the local setup and formulate a twisted AFL. In Section \ref{section: special cycles and RZ}, we introduce and study mirabolic special cycles on Rapoport--Zink spaces. In Section \ref{section: twisted orbits, FL}, we formulate the matching of orbits and orbital integrals and recall the twisted fundamental lemma. In Section \ref{section: TAFL}, we introduce twisted fixed cycles and formulate our twisted AFL. In Section \ref{section: arithmetic induction}, we introduce arithmetic induction systems and new relative Cayley maps, and show both sides of the twisted AFL are such examples using mirabolic special cycles. In Section \ref{section: global special cycles}, we introduce global Kudla--Rapoport divisors and twisted complex multiplication cycles on unitary Shimura varieties. In Section \ref{section: local-global}, we form global analytic and geometric generating functions and prove local-global decomposition formulas for non-singular Fourier coefficients via basic uniformizations. In Section \ref{section: final proof}, we establish the twisted AFL. In Section \ref{section: TAGGP}, we formulate twisted arithmetic Gan-Gross-Prasad conjecture and indicate its relation to our twisted AFL.

\subsection*{Acknowledgements}
I thank M. Rapoport, C. Li, A. Mihatsch and W. Zhang for helpful discussions. I am also grateful for the hospitality of the Simons Laufer Mathematical Sciences Institute in Spring 2023, where the project began.

\subsection{Notations and Conventions}
\begin{itemize}
\item  Let $X$ be an algebraic variety with an action of an algebraic group $H$ over a field $F_0$. We say $x \in X(F_0)$ is called regular semi-simple if the stabilizer of $x$ inside $G$ is trivial, and the orbit of $x$ is closed in $X$. We often denote by $(-)_\rs$ (resp. $[-]_\rs$) the space of regular semi-simple elements (resp. orbits).
	\item Let $p$ be an odd prime.
	\item Let $F_0/\BQ_p$ be a finite extension. Choose a uniformizer $\varpi$ of $F_0$.  Denote by $\breve{F_0}$ the $p$-adic completion of a maximal unramified extension of $F_0$ and $O_{\breve{F_0}}$ be its ring of integers.
        \item Denote by $\BF_q$ (resp. $\BF$) the residue field of $F_0$ (resp. $\breve{F_0}$). Let $\sigma_{\breve{F_0}/F_0} \in \Gal(\breve{F_0}/F_0)$ be the canonical lifting of $q$-Frobenius $x \rightarrow x^q$ on $\BF$. 
        \item $F/F_0$ is an unramified quadratic extension of $p$-adic local fields. We fix an embedding $F \to \breve{F_0}$. Denote by $\overline{(-)} \in \Gal(F/F_0)$ the non-trivial Galois involution.
        \item Denote by $\val=\val_{F_0}: F^\times_0 \to \BZ$ the standard valuation of $F_0$, in particular $\val_{F_0}(\varpi)=1$. Let $\eta=\eta_{F/F_0}$ be the quadratic character $\eta(x)=(-1)^{\val_{F_0}(x)}: F^{\times}_0 \to \{\pm 1\}$. 
    \item  Let $t \in \BR$ and $n \in \BZ$. For $h \in \SL_2(\BR)$ under the Iwasawa decomposition 
		$$h=\begin{pmatrix}
			1 & b \\ 0 & 1
		\end{pmatrix}\begin{pmatrix}
			a^{1/2} & 0 \\ 0 & a^{-1/2}
		\end{pmatrix} \kappa_\theta, \quad a \in \BR_{+}, \quad b \in \BR, $$ 
		and 
		$$
		\kappa_\theta=\begin{pmatrix}
			\cos \theta & \sin \theta \\ -\sin \theta & \cos \theta
		\end{pmatrix} \in \mathrm{SO}_2(\BR),
		$$
The standard weight $n$ Whittaker function on $h \in \SL_2(\BR)$ is defined as		\begin{equation}\label{defn: Whittaker SL2}
			W_{t}^{(n)} (h):=|a|^{n/2} e^{2\pi i t (b+a i)} e^{in\theta}.
		\end{equation}
  \item In global set up, denote by $F_0$ a totally real number field. Let $F_{0,+}$ be the cone of totally positive elements in $F_0$. Let $F/F_0$ be a CM quadratic extension. Denote by $\Phi$ a CM type of $F$. The CM type $\Phi$ is called unramified at $p$, if $\Phi \otimes \BQ_p: F \otimes \BQ_p \to \ov \BQ_p$ is induced from a CM type of the maximal unramified (over $\BQ_p$) subalgebra of $F \otimes \BQ_p$.
  \item  In global set up, we denote by $\BA$, $\BA_{0}$ for the adele rings of $\BQ$ and $F_0$ respectively. Denote by $(-)_f$ the terms of finite adeles. Denote by $(-)_p$ (resp. $(-)_v$) the local term at a place $p$ of $\BQ$ (resp. a place $v$ of $F_0$). Let $\Delta$ be a finite collection of places of $F_0$. Denote by $(-)_{\Delta}$ the product of $(-)_v$ over all places $v \in \Delta$. 
  \item 
  In global set up, for Weil representations, we use the additive character  $\psi_{F_0}: F_0 \backslash \BA_{F_0} \rightarrow \mathbb C$ given by $\psi_{F_0}:=\psi_\BQ \circ \mathrm{Tr}_{F_0/\BQ}$, where $\psi_{\BQ}: \BQ \backslash \BA  \rightarrow \BC$ is the standard additive character. 
  \item For a compact open subgroup $K_0 \subseteq \SL_2(\BA_{0, f})$, denote by $\CA_{\rm hol}(\SL_2(\BA_{F_0}), K_0, n)_{\ov{\BQ}}$ the space of holomorphic automorphic forms on $\SL_2(\BA_{F_0})$ that are left invariant under $\SL_2(F_0)$, right invariant under $K_0$, of parallel weight $n$ and with Fourier coefficients in the algebraic closure $\ov{\BQ}$ of $\BQ$.
\end{itemize}

\section{Mirabolic special cycles on Rapoport--Zink spaces}\label{section: special cycles and RZ}

Let $F/F_0$ be an unramified quadratic extension of $p$-adic local fields. For $n \geq 1$, there are two isomorphism classes of $n$-dimensional $F/F_0$ hermitian spaces, and we denote by $V$ the split one and $\BV$ the non-split one.

\subsection{The Rapoport--Zink spaces}

Let $r, s \geq 0$ be two integers such that $r+s=n$. For any $\Spf O_{\breve F_0}$-scheme $S$, \emph{a basic tuple $(X, \iota)$ of signature $(r, s)$} over $S$ is the following data:
\begin{itemize}
	\item
	$X$ is a formal $O_{F_0}$-module over $S$ of relative $O_{F_0}$-height $2n$ and dimension $n$ which is strict, i.e. the induced action of $O_{F_0}$ on $\Lie X$ is via the structure morphism $O_{F_0} \to \CO_S$. 
	\item $\iota: O_F \to \End(X)$ is an $O_F$-action on $X$ that extends the $O_{F_0}$-action such that the \emph{Kottwitz condition} of signature $(r,s)$ holds for all $a \in O_{F} $:
	\begin{equation}\label{eq: Kottwitz}
		\charpol (\iota(a)\mid \Lie X) = (T-a)^r(T-\ov a)^{s} \in \CO_S[T].
	\end{equation} 
\end{itemize}

\emph{A principally polarized basic tuple of signature $(r, s)$} is a tuple $(X, \iota, \lambda)$ where $(X, \iota)$ is a basic tuple of sign $(r,s)$ and $\lambda: X \to X^\vee$ is a principal polarization that is $O_F/O_{F_0}$ semi-linear, i.e. the Rosati involution $\mathrm{Ros}_\lambda$ on $\iota: O_{F} \to \End(X)$ agrees with $a \mapsto \overline{a}$. An isomorphism $(X_1, \iota_1) \cong (X_2, \iota_2)$ (resp. $(X_1, \iota_1, \lambda_1) \cong (X_2, \iota_2, \lambda_2)$) between two such basic tuples (resp. principally polarized basic tuples) is an $O_F$-linear isomorphism $\varphi: X_1 \cong X_2$ (resp. $\varphi: X_1 \cong X_2$ such that $\varphi^*(\lambda_2)=\lambda_1$).

Up to $O_F$-linear quasi--isogenies, there exists a unique such tuple $(\BX, \iota_{\BX}, \lambda_{\BX})$ of sign $(n-1,1)$ over $\BF$. Fix one schoice of $(\BX, \iota_{\BX}, \lambda_{\BX})$ as the {\em framed principally polarized basic tuple}. 

\begin{definition}\label{defn: local GL RZ space}
	The general linear Rapoport--Zink space is the functor 
	$$\CN_{n}^{\GL} \to \Spf O_{\breve F_0} $$
	sending $S$ to the set of isomorphism classes of tuples $(X, \iota, \rho)$, where
	\begin{itemize}
		\item $(X, \iota)$ is a basic tuple of signature $(n-1,1)$ over $S$.
		\item $\rho: X \times_{S} {\ov S} \to \BX \times_\BF \ov S $ is an $O_F$-linear quasi-isogeny of height $0$ over $\ov S:=S \times_{O_{\breve F_0}} \mathbb F$.
	\end{itemize}
\end{definition}

\begin{definition}\label{defn: local RZ space}
The unitary Rapoport--Zink space is the functor 
	$$\CN_{n} \to \Spf O_{\breve F_0} $$
	sending $S$ to the set of isomorphism classes of tuples $(X, \iota, \lambda, \rho)$, where
	\begin{itemize}
		\item $(X, \iota, \lambda)$ is a principally polarized basic tuple of signature $(n-1, 1)$ over $S$.
		\item $\rho: X \times_{S} {\ov S} \to \BX \times_\BF \ov S $ is an $O_F$-linear quasi-isogeny of height $0$ over the reduction $\ov S:=S \times_{O_{\breve F_0}} \mathbb F$ such that $\rho^*(\lambda_{\BX, \ov S}) = \lambda_{\ov S}$.
	\end{itemize}
\end{definition}

\begin{proposition}
The functor $\CN_n^{\GL}$ (resp. $\CN_n$) is representable by a formal scheme locally formally of finite type, which is formally smooth over $\Spf O_{\breve F_0}$ of relative dimension $2(n-1)$ (resp. $n-1$). 
\end{proposition}
\begin{proof}
The representability follows from the method of Rapoport--Zink \cite[Thm. 2.16]{RZ96}. The formal smoothness and dimension formulas follow from Grothendieck--Messing theory.
\end{proof}

\subsection{Symmetries and involutions on Rapoport--Zink spaces}

Let $\BE$ be a formal $O_{F_0}$-module over $\BF$ of relative height $2$ and dimension $1$. Complete $\BE$ into a principally polarized framing basic triple $(\BE, \iota_{\BE}, \lambda_{\BE})$ of \emph{signature $(1, 0)$} over $\BF$. Denote by $(\CE, \iota_\CE, \lambda_\CE)$ the basic principally polarized triple over $\Spf O_{\breve F_0}$ obtained from the canonical lifting of $\BE$.

Let \emph{the space of special quasi-homomorphisms} be the $F$-vector space 
\begin{equation}\label{eq: space of special quasi-homos}
\Hom_{O_F}^{\circ}(\BE, \BX) := \Hom_{O_F}(\BE, \BX) \otimes \BQ
\end{equation}
equipped with the hermitian form:
\begin{equation}
	(x,y):= \lambda_{\BE}^{-1} \circ y^{\vee} \circ \lambda_{\BX} \circ x \in \Hom_{O_F}^{\circ}(\BE, \BE) \cong F, \quad \quad \forall x, \, y \in \Hom_{O_F}^{\circ}(\BE, \BX).
\end{equation}
Then $\Hom_{O_F}^{\circ}(\BE, \BX)$ is a non-split hermitian space of dimension $n$ and we may write $\BV=\Hom_{O_F}^{\circ}(\BE, \BX)$. The linear dual of $\BV$ is $\BV^*= \Hom_{O_F}^{\circ}(\BX, \BE)$. The hermitian form on $\BV$ gives an isomorphism $\lambda_\BV: \BV \to \BV^*$, hence an involution $\sigma(-)=\lambda_\BV^{-1} \circ (-)^* \circ \lambda_\BV : \GL(\BV) \to \GL(\BV)$ with fixed subgroup $\U(\BV)$. There is an identification $\GL(\BV)(F_0) \cong \Aut^\circ(\BX, \iota_{\BX})$ where $\sigma$ is identified with 
\[
\sigma(g)= \lambda_\BX^{-1} \circ (g^\vee)^{-1} \circ \lambda_\BX, \quad \forall g \in \Aut^\circ(\BX, \iota_{\BX}).
\]

The general linear group $\GL(\BV)(F_0) \cong \Aut^\circ(\BX, \iota_{\BX})$ acts on $\CN_n^\GL$ by acting on the framing:
\[
g \cdot (X, \iota, \rho) = (X, \iota, g \circ \rho). 
\]
Similarly, the unitary group $\U(\BV)(F_0) \cong \Aut^\circ(\BX, \iota_{\BX}, \lambda_{\BX})$ acts on $\CN_n$ by acting on the framing. By Drinfeld's rigidity of quasi-isogenies, we have a $\U(\BV)(F_0)$-equivariant embedding 
\[
i_{\mathrm{BC}}: \CN_n \to \CN_n^\GL.
\]

\begin{definition}
	Consider the Galois involution 
	$$\sigma: \CN^\GL_n \to \CN^\GL_n
	$$
	sending $(X, \iota, \rho) \to (X^\vee,  \overline{\iota^\vee}, \lambda_\BX^{-1} \circ (\rho^\vee)^{-1} ).$
\end{definition}

\begin{proposition}
	The action of $\GL(\BV)(F_0)$ on $\CN^\GL_n$ is $\sigma$-equivariant. The fixed point locus of $\sigma$ on $\CN^\GL_n$ recovers $\CN_n$:
	\[
	(\CN^\GL_n)^{\sigma=\id} = \CN_n.
	\]
\end{proposition}
\begin{proof}
We have $\sigma (g. (X, \iota, \rho)) = (X^\vee, \overline{\iota^\vee},  \lambda_\BX^{-1} \circ ((g \circ \rho)^\vee)^{-1})= (X^\vee, \overline{\iota^\vee},  \lambda_\BX^{-1} \circ (g^\vee)^{-1} \circ (\rho^\vee)^{-1}) = \sigma(g) (\sigma(X, \iota, \rho)).$Hence the $\GL(\BV)(F_0)$-action on $\CN^\GL_n$ is $\sigma$-equivariant. If $(X, \iota) \in (\CN^\GL_n)^{\sigma=\id}$, then there is an $O_F$-linear isomorphism $\varphi: X \cong X^\vee$ compatible with the framing $\rho$ and $\lambda_\BX^{-1} \circ (\rho^\vee)^{-1}$. This is equivalent to that $\varphi$ is a lifting of $\lambda_\BX$ to $X$. Hence $\varphi$ is a principal polarization on $X$ and we see that $(X, \iota, \varphi) \in \CN_n$. Hence $(\CN^\GL_n)^{\sigma=\id} = \CN_n$. 
\end{proof}

\begin{remark}
The above proposition can be regarded as a geometrization of base change and Galois descent for representations of $\U(V)$ and $\GL(V)$.
\end{remark}

\subsection{Geometry of mirabolic special cycles and Kudla--Rapoport cycles}

Consider a non-zero vector $u \in \BV - 0$. The Kudla--Rapoport cycle \cite{KR2011local}
$$
\CZ(u) \to \CN_n
$$ is the closed formal subscheme of $\CN$ sending a $\Spf O_{\breve{F_0}}$-scheme $S$ to the subset $(X, \iota_X, \lambda_X, \rho_X) \in \CN(S)$ such that
	\[ (\rho_{X})^{-1}  \circ u: \BE \times_{\BF} \ov S \to X \times_{S} \ov S  \]  
	lifts to a homomorphism from $\CE$ to $X$ over $S$. By \cite[Prop 3.5]{KR2011local}, $\CZ(u)$ is a relative Cartier divisor in $\CN_n$.

We now introduce two kinds of \textit{mirabolic special cycles} on $\CN^\GL_n$, which recover the notion of Kudla--Rapoport cycles after base change to $i_{\mathrm{BC}}: \CN_n \to \CN^\GL_n$.

\begin{definition}
For $u \in \BV - 0$, the mirabolic special cycle
	$$
	\CZ^{\GL}(u) \to \CN^\GL_n
	$$ is the subfunctor of $\CN^\GL$ sending a $\Spf O_{\breve{F_0}}$-scheme $S$ to $(X, \iota_X, \rho_X) \in \CN^\GL_n(S)$ such that
	\[ (\rho_{X})^{-1}  \circ u: \BE \times_{\BF} \ov S \to X \times_{S} \ov S  \]  
	lifts to a homomorphism from $\CE$ to $X$ over $S$. For $u^* \in \BV^* -0 $,  the mirabolic special cycle
	$$
\CZ^{\GL^*}(u^*) \to \CN^\GL_n
$$ is the subfunctor of $\CN^\GL$ sending a $\Spf O_{\breve{F_0}}$-scheme $S$ to $(X, \iota_X, \rho_X) \in \CN^\GL_n(S)$ such that $u^* \circ \rho_{X} :  X \times_{S} \ov S \to \BE \times_{\BF} \ov S$  lifts to a homomorphism from $X$ to $\CE$ over $S$.
\end{definition}

\begin{theorem}\label{prop: geometry of mirabolic cycle}
\begin{enumerate}
\item The subfunctors $\CZ^\GL(u)$ and $\CZ^{\GL^*}(u^*)$ are represented by formal schemes, and are relative Cartier divisors in $\CN_{n}^\GL$.
\item (The pullback formula) The diagram
      \[
\xymatrix{
	\CZ(u)  \ar[r] \ar[d] &  \CZ^\GL(u)  \ar[d]  \\  
	\CN_n  \ar[r] & \CN^\GL_n }
\]
is Cartesian, in other words $\CZ^\GL(u)|_{\CN_n} = \CZ(u)$.
	\item (Galois duality) $
	\sigma(\CZ^\GL(u^*))= \CZ^\GL(u), \, \mathrm{if} \, \lambda_\BX^{-1} \circ (u^*)^\vee \circ \lambda_\BE =  u.$
    \item (Galois descent) $[\CZ^\GL(u)]^{\sigma=\id}=\CZ(u)$.
	\item (Translation laws) For $a \in O_F^\times$, $\CZ^\GL(au)=\CZ^\GL(u)$. For $g \in \GL(\BV)(F_0)$, we have
   \[
g. \CZ^\GL(u)= \CZ^\GL(gu), \, \,  g. \CZ^{\GL^*}(u^*)= \CZ^{\GL^*}(gu^*).
    \]
	\item (Cancellation laws) If $(u, u^*)=1$, then $\CZ(u) \cap \CZ(u^*) \cong \CN_{n-1}^\GL$.
    \item (Functoriality) Let $n,m \geq 1$. Consider a natural embedding $\CN_n^\GL \times \CN_m^\GL \to \CN_{n+m}^\GL$ given by taking direct sums. For $(u_1, u_2) \in \BV_n \oplus \BV_m =\BV_{n+m}$ such $u_i \not =0$, we have 
    \[
    \CZ^\GL_{n+m}(u_1,u_2)|_{\CN_n^\GL \times \CN_m^\GL}= \CZ^\GL_{n}(u_1) \times \CZ^\GL_m(u_2).
    \]
\end{enumerate}
\end{theorem}
\begin{proof}
We prove the first part following the proof of \cite[Prop 3.5]{KR2011local}. The first part follows from the proof that Rapoport--Zink spaces are represented by formal schemes in \cite[Prop 2.9]{RZ96}.   

We now prove that $\CZ^\GL(u)$ is locally defined by the vanishing of one equation. Let $R$ be a $\Spf O_{\breve{F}}$-algebra and $I$ is an ideal of $R$ such that $I^2 = 0$. We equip $I$ with trivial divided powers and let $R_0 = R/I$. Given a morphism $\phi: \Spec R \to \CN^\GL_n$ whose restriction to $\Spec R_0$ factors through  $\CZ^\GL(u)$. Then we need to show the obstruction to factoring the given morphism $\phi$ through $\CZ^\GL(u)$ is given by the vanishing of one element in $I$. Let $X$ be the $A$-point of $\CN$ corresponding to $\phi$. The condition implies that there is an $O_F$-linear homomorphism $\alpha: D_\CE \to D_X$ of relative Dieudonne crystals over $\Spec A$, which maps the Hodge filtration $\CF_\CE $ to $\CF_X$ after tensoring with $A_0$. By Grothendieck--Messing theory, the obstruction is exactly the condition $\alpha(\CF_\CE) \subseteq \CF_X$. As $\alpha$ is $O_F$-linear and $\CF_\CE$ is generated by one element $\bar{1}_0$ as $O_F \otimes \BA$ module, the obstruction is equivalent to that $\alpha(\bar{1}_0) \in (F_X)_0$, i.e. $\alpha(\bar{1}_0)=0 \in (\Lie X)_0$. By the Kottwitz signature condition, $(\Lie X)_0$ is a locally free $A$-module of rank $1$. After choosing a local generator of $\Lie X_0$, we may identify $\alpha(\bar{1}_0)$ with an element in $A$ with zero image in $A_0$, i.e. an element in $I$. This shows that $\CZ^\GL(u)$ is a divisor.

To show that $\CZ^\GL(u)$ is a relative Cartier divisor, we only need to show any local generator of the divisor is non-zero and not divisible by $p$. As $\CZ^\GL(u)|_{\CN_n} = \CZ(u)$, this follows from the corresponding proposition for $\CZ(u)$ proved in \cite[Prop 3.5]{KR2011local}.

The following four properties follow by definition. For the cancellation law, if $(u,u^*)=1$ we can obtain a splitting of $\BX=\BX^\flat \oplus \BE$ where $\BX^\flat=\Ker u^*: \BX \to \BE$ and the direct summand $\BE$ is the image of $u$. Since $u, u^*$ lift to $\CZ(u) \cap \CZ(u^*)$, similarly we obtain a splitting of the universal tuple $X=X^\flat \oplus \CE$ over  $\CZ(u) \cap \CZ(u^*)$. Therefore $\CZ(u) \cap \CZ(u^*) \cong \CN_{n-1}^\GL$. Finally, the functoriality follows by moduli definitions.
\end{proof}

For any subset $S \subseteq \BV$, we define mirabolic special cycle $\CZ(S)^\GL \subseteq \CN^\GL_n$ as the lifting locus of $S$.

\begin{remark}[Special linear cycles on Hodge type Rapoport-Zink spaces] \label{defn: pullback mirabolic speical cycles}
Let $F/F_0$ be a quadratic extension. Let $V$ be a $n$-dimension $F$-vector space. We extend the notion of mirabolic special cycles to any (relative) Rapoport-Zink spaces $\CN_V:=\CN_{\Res_{F/F_0}\GL(V), b_V, \mu_V, K_p}$ of EL type over $F_0$. For simplicity, assume $K_p \leq \GL(V)$ is hyperspecial. 

Let $\BX$ be a framing object for $\CN_{V}$, and $\BE$ be a formal Lubin-Tate $O_F$-module of dimension $1$ (and $O_F$-height 1) over $\BF$ with its canonical lifting $\CE$ over $O_{\breve{F}}$. Define the space of special quasi-homomorphism to be the $F$-vector space $\BV=\Hom_{O_F}(\BE, \BX) \otimes \BQ$. Then for non-zero $u\in \BV$ we define
\[
Z_V(u) \to \CN_V
\]
as the closed formal subscheme of lifting locus of $u$ to $\CE \to \mathcal{X}$. Because of the $O_F /O_{F_0}$-signature condition given by $\mu$, it is still representable by a formal subscheme. Via pullback, we get \emph{special linear cycles} 
$$
Z_{V}(u) \to \CN_G:=\CN_{G, b, \mu, K}
$$ on general unramified Hodge type Rapoport-Zink spaces $(G, b, \mu)$ \cite{HP17}\cite{kim2018} associated to a representation $V$ of $G$ which induces a embedding of local Shimura datum $(G, b, \mu) \to (\GL(V), b_V, \mu_V)$. 
\end{remark}

\section{Twisted orbits, orbital integrals and fundamental lemmas}\label{section: twisted orbits, FL}

In this section, we introduce orbits and orbital integrals in our twisted setup. We recall the twisted fundamental lemma, which is a motivation to the twisted arithmetic fundamental lemma, and will be used in our global set up. 

\subsection{Orbits and orbital integrals}

On the unitary side, let $(V, h_V)$ be a $n$-dimensional hermitian space over $F$. Consider reductive groups over $F_0$:
$$
H_V=\U(V), \, G_V=\GL(V).
$$
Let $X_V=\Herm(V, h_V)$ be the set of hermitian structures $A$ on $V$ such that $(V, A) \cong (V, h_V)$. We identify 
\begin{equation}
G_V/H_V \cong X_V, \quad g \mapsto A_g:=gh_Vg^{-1}.
\end{equation}

Let $Y_V=V$. Consider the action of $g \in H_V$ on 
$$X_V \times Y_V = \Herm(V, h_V) \times V$$ 
given by $h. (g, u)= (h^{-1}gh, h^{-1}u)$. The pair $(g, u)$ is regular semi-simple (with respect to $H_V$) if and only if $u, A_gu, \hdots, A_g^{n-1}u$ are linearly independent. Choose a Haar measure on $H_V(F_0)$.
For $\Phi \in \CalS(X_V \times Y_V)$, define its orbital integral on $(g, u) \in X_V \times Y_V$ by
\begin{equation}
\Orb((g, u), \Phi):=\int_{H_V(F_0)} \Phi(h.(g,u)) dh
\end{equation}
which converges absolutely if $(g, u)$ is regular semi-simple. 

On the general linear side, let $V_0$ be a $n$-dimensional vector space over $F_0$. Let
$$
H'=\GL(V_0), \, G'=\GL(V_0) \times \GL(V_0), \, H'_{20}= \GL(V_0).
$$
We consider diagonal embeddings $H' \rightarrow G' \leftarrow H'_{20}$. The space $X'=G'/H'_{20}$ can be naturally identified with $\GL(V_0)$, where the natural left action of $H'$ on $X'$ is identified with conjugacy action on $\GL(V_0)$. Let $Y'=V_0':=V_0 \times V_0^*$ with the natural action of $H'$.

Similarly, we consider the action of $H'$ on 
$$X' \times Y'= \GL(V_0) \times V_0 \times V_0^*$$
given by $h. (\gamma, u'=(u_1,u_2))= (h^{-1}\gamma h, h^{-1}u_1, u_2h)$. We use the \emph{transfer factor}
\begin{equation}\label{transfer factor: semi-Lie}
	\omega(\gamma, u') :=(-1)^{\val_{F_0} (\det (\gamma^{i} u_1 )_{i=0}^{n-1}) } \in \{\pm 1\}.
\end{equation}

Choose a Haar measure on $H'(F_0)$. For $\Phi' \in \CalS(X' \times Y')$, define the $s$-variable orbital integral ($s \in \BC$)
\begin{equation}
\Orb((\gamma, u'), \Phi', s):= \omega(\gamma, u') \int_{H'(F_0)} \Phi'(h.(\gamma, u')) |h|^s \eta(h) dh
\end{equation}
which converges absolutely if $(\gamma, u')$ is regular semi-simple. Here we write $|h|=|\det h|, \eta(h)=\eta(\det h)$ for short. Set 
\[
\Orb((\gamma, u'), \Phi')= \Orb((\gamma, u'), \Phi', 0), \quad \partial \Orb ((\gamma, u'), \Phi')= \frac{d}{ds}|_{s=0} \Orb((\gamma, u'), \Phi', s).
\]

Consider the embedding $X_V \times V   \to \End(V) \times V \times V^{*}$ given by $(g,u) \to (g, \, u, \, u^*:=(x \mapsto h_V(x,u)))$. We may choose a basis of $V$ and assume that $V=V_0 \otimes_{F_0} F$.

\begin{definition}\label{defn: semi-Lie match elements}
A pair $(g,u) \in X_V \times Y_V$ matches a pair $(\gamma, u') \in  X' \times Y'$ if they are conjugated by $\GL(V)$ inside $\End(V) \times V \times V^*$.
\end{definition} 
The matching relation is equivalent to the matching of the following invariants:
\[
\det(T \id_{V}+g)=\det(T \id_{V}+\gamma) \in F'[T], \quad (g^iu,u)=u_2 (\gamma^iu_1), \quad 0 \leq i \leq n-1.
\]	

\begin{proposition}\label{matching of orbits: Un to GLn}
The matching relation gives a natural bijection of regular semi-simple orbits:
\[
[H' \backslash (X' \times Y')]_\rs \cong [H_V \backslash (X_V \times Y_V)]_\rs \coprod [H_{\BV} \backslash (X_\BV \times Y_\BV)]_\rs, \, \, (\gamma,u') \leftrightarrow (A, u).
\]
\end{proposition}
\begin{proof}
This follows from the same proof in the Jacquet-Rallis case. See \cite[Section 4.1]{leslie2019endoscopic}.
\end{proof}

\subsection{Transfers and Twisted Fundamental lemmas}

Choose a self-dual lattice $L$ in $V$. Choose an orthogonal basis $\{e_i\}$ of $L$, and let $L_0= \sum_{i} O_{F_0}e_i$. We may assume that $V_0=L_0 \otimes \BQ$. We consider the lattice $L'=L_0 \times L_0^* \subseteq Y'=V_0'$. We normalize the Haar measure on $\U(V)(F_0)$ (resp. $\GL(V_0)(F_0)$) such that the stabilizer $\U(L)$ (resp. $\GL(L_0)$) is of volume $1$.

\begin{definition}
We say $\Phi' \in \CS(X' \times Y')$ and a pair $(\Phi_V \in \CS(X_V \times Y_V)), \Phi_{\BV} \in \CS(X_\BV \times Y_\BV)$ are \emph{transfers} to each other, if for any regular semisimple pair $(\gamma, u') \in X' \times Y'$, we have 
\[
\Orb((\gamma,u'), \Phi')= \begin{cases}
   \Orb((g,u), \Phi_V), \quad \text{if $(\gamma,u')$ matches $(g,u) \in (X_V \times Y_V)_\rs$.}\\
   \Orb((g,u), \Phi_\BV), \quad \text{if $(\gamma,u')$ matches $(g,u) \in (X_\BV \times Y_\BV)_\rs$.}  \\
\end{cases}
\]
\end{definition}

\begin{remark}
Let $E_0/F_0$ be a \'etale quadratic algebra. Let $E=F \otimes_{F_0} E_0$. Let $F'/F_0$ be the third quadratic extension inside $E$. We can consider the following generalization of the above orbits ($E_0=F$) and the Jacquet--Rallis case ($E_0=F_0 \times F_0$). Let  
\[
H_V=\U(V), \quad G_{V}^{E_0}=\U(V_{E_0}).
\]
Instead of $X_V$, consider the hermitian symmetric space 
\[
X_{V}^{E_0}=\U(V_{E_0})^{-}:=\{ g \in \U(V_{E_0}) | g \sigma_{E/F}(g)=\id. \}
\] 
Let 
\[
H'=\GL(V_0), \quad G^{'E_0} = \GL(V_0,F'), \quad H'_{20} = \GL(V_0).  
\]
Instead of $X'$, consider the $F'/F_0$-symmetric space over $F_0$:
\[
X^{'E_0}=S_{n,F'/F_0}=\{ \gamma \in \GL_n(F') | \gamma \sigma_{F'/F_0} (\gamma) =\id \}.
\]
By Hilbert 90, we have $\U(V_{E_0})^{-} \cong \U(V_{E_0})/\U(V), \, \, S_{n,F'/F_0}=\GL_n(F')/\GL_n(F_0).$ Note that the Lie algebra of $X_{V, E_0}$ and $X'_{E_0}$ can be identified after multiplying suitable elements in $F'$. Then we have a similar matching of regular semi-simple orbits
\[
[H' \backslash (X^{'E_0} \times Y')]_\rs \cong [H_V \backslash (X_V^{E_0} \times Y_V)]_\rs \coprod [H_{\BV} \backslash (X_\BV^{E_0} \times Y_\BV)]_\rs, \, \, (\gamma,u') \leftrightarrow (A, u).
\]
similar to \cite[Section 4.1]{leslie2019endoscopic}.
\end{remark}

Let $\GL(L_0)$ be the stabilizer of $L_0$ in $\GL(V_0)$. Let $\Herm(L)$ be the subset of $A \in X_V=\Herm(V, h_V)$ such that $L$ is integral under $A$. Define standard test functions
\begin{equation} \label{standard test function}
\Phi_{L'}'=1_{\GL(L_0)} \times 1_{L'} \in C_c^\infty(X' \times Y'), \, \Phi_L= 1_{\Herm(L)} \times 1_{L} \in C_c^\infty(X_V \times Y_V)
\end{equation}

We have the following twisted fundamental lemma which appears in the work of Danielle Wang \cite{Wang-TGGP}.

\begin{theorem}[Twisted Fundamental Lemmas] \label{Twisted FL}
$\Phi'_{L'}$ is a transfer of $(\Phi_L, 0)$.
\end{theorem}	
\begin{proof}
Using  the Cayley map, we reduce to similar transfer results on lie algebra case, which is equivalent to the Jacquet--Rallis fundamental lemma proved in \cite{beuzart2021new}.
\end{proof}

\section{Twisted arithmetic fundamental lemmas}\label{section: TAFL}

In this section, we introduce twisted derived fixed cycles and formulate a twisted arithmetic fundamental lemma. Recall we have a $\U(\BV)$-equivariant embedding 
\[
\CN_n \to \CN_n^\GL. 
\]
\begin{definition}\label{twisted fixed cycles}
For $g \in \GL(\BV)(F_0)$, the twisted fixed cycle of $g$ is the derived special cycle (in the Grothendieck group of $\CN_n$) which only depends on $g \in X_\BV=\GL(\BV)/\U(\BV)$:
\begin{equation}
\CN_n^{\Herm}(g)= g \CN_n \cap^{\BL}_{\CN^\GL_n} \CN_n \in K_0'(\CN_n),
\end{equation}    
\end{definition}
which is a derived $1$-cycle by \cite[Lemma B.2]{AFL-Wei2019}. We identify $X_\BV=\GL(\BV)/\U(\BV)=\Herm(\BV, h_\BV)$ as the set of hermitian structures $A$ on $\BV$ such that $(\BV, A) \cong (\BV, h_\BV)$. 

\begin{proposition}\label{int proper scheme}
For a regular semi-simple pair $(g, u) \in \GL(\BV)/\U(\BV) \times \BV$, the intersection $g \CN_n \cap_{\CN^\GL_n} \CZ(u)$ is a proper scheme.
\end{proposition}
\begin{proof}
Let $(X, \iota, \lambda) \in (g \CN_n \cap_{\CN^\GL_n} \CN_n)(S)$. Then there exists an $O_{F}$-linear morphism $\varphi: X \to X$ lifting $g$. If $(X, \iota, \lambda) \in  g \CN_n \cap_{\CN^\GL_n} \CZ(u)$, using the lifting $\varphi$ we see that $(X, \iota, \lambda) \in \CZ(L_{g,u})$ where $L_{g,u}$ is the $O_F$-span of $u, gu, \hdots, g^{n-1} u$. As $(g,u)$ is regular semi-simple, $L_{g,u}$ has full rank $n$ hence $\CZ(L_{g,u})$ is a proper scheme by \cite[Lemma 2.10.1]{LiZhangKR2019}. As $g\CN_n \cap_{\CN^\GL_n} \CZ(u)$ is a a closed formal subscheme of $\CZ(L_{g,u})$, it is also a proper scheme.
\end{proof}

For any regular semi-simple pair $(g,u) \in (X_\BV \times \BV)(F_0)_\rs$, consider the derived intersection number
\begin{equation} \label{defn: intersection number}
\Int^{\Herm, \BV}(g, u) = \CN_n^{\Herm}(g) \cap^\BL_{\CN_n} \CZ(u) = g\CN_n \cap^{\BL}_{\CN^\GL_n} \CZ(u) = \chi(\CN_n^\GL,  \CO_{\CN_n} \otimes^{\BL} \CO_{\CZ(u)} ).
\end{equation}

By Proposition \ref{int proper scheme}, $\Int(g, u) \in \BZ$ is well-defined and finite. 
\begin{proposition}\label{prop: IntHerm=mirabolic int}
We have 
\[
\Int^{\Herm, \BV}(g, u) = \CN_n \cap^\BL_{\CN_n^\GL} g \CN_n \cap^\BL_{\CN^\GL_n} \CZ^\GL(u).
\]
\end{proposition}
\begin{proof}
As $\CZ^\GL(u), \CZ(u)$ are Cartier divisors, the intersection $\CZ(u)=\CZ^\GL(u) \cap_{\CN^\GL_n} \CN_n$ has correct codimension hence is equal to $\CZ^\GL(u) \cap_{\CN^\GL_n}^\BL \CN_n$ by \cite[Lemma B.2]{AFL-Wei2019}. The result follows.
\end{proof}

Now we formulate our twisted arithmetic fundamental lemma.

\begin{conjecture}[Twisted arithmetic fundamental lemma] \label{Twisted AFL conjecture}
For any $(g,u) \in (\GL(\BV)/\U(\BV) \times \BV)(F_0)_\rs$ matching $(\gamma,u') \in (\GL_n \times V_0')(F_0)_\rs$, we have
\[			
 \partial \Orb((\gamma,u'),  \Phi'_{L'}) =- \Int^{\Herm, \BV}(g,u) \log q.
\]
\end{conjecture}

\begin{remark}
The original arithmetic fundamental lemma \cite{AFL-Invent} study similar arithmetic intersection problems via the embedding of $\CN_n \to \CN_n \times \CN_n$ and graphs of automorphisms on $\CN_n \times \CN_n$. In above conjecture, we use the embedding $\CN_n \to \CN_n^\GL$, hence we call it a twisted arithmetic fundamental lemma.
\end{remark}

\begin{proposition}
The twisted AFL conjecture \label{Twisted AFL conjecture} holds when $n=1$.
\end{proposition}
\begin{proof}
In this case, the (twisted) fixed cycles will be the whole $\CN_1$. The twisted AFL agrees with the original AFL \cite{AFL-Invent} and follows from Gross's theory of quasi-canonical lifting.
\end{proof}

\begin{proposition}\label{TAFL: maximal order}
The twisted AFL conjecture \label{Twisted AFL conjecture} holds in the maximal order case, i.e., when $O_F[g]$ ($g \in \Herm(\BV, h_\BV)$)is a product of discrete valuation rings.
\end{proposition}
\begin{proof}
We follow the proof in the AFL case \cite[Section 9]{Mihatsch-comparison}\cite[Prop 2.6]{AFL-Wei2019}. Write $F[g]=\prod_{i=1}^m F_i$ as a product of fields. Write $E_0=F_0$. Then above identities is reduced to the twisted AFL in the case $n=1$ for $F_i/F_{i,0}$ (and $E_{i,0}= F_{i,0}$). 
\end{proof}

In the rest of the paper, we will prove the following main theorem.

\begin{theorem}[Twisted AFL]\label{Twisted AFL theorem}
The twisted AFL conjecture \label{Twisted AFL conjecture} holds for any regular semi-simple $(g,u)$ and $n \geq 1$.
\end{theorem}

\section{Arithmetic induction systems and reductions}\label{section: arithmetic induction}

In this section, we develop and apply the idea of arithmetic induction systems (Definition \ref{defn: arithmetic induction system}). We show the geometric side and analytic side in the twisted AFL conjecture \ref{Twisted AFL conjecture} are both arithmetic induction systems, which will be used in our proof of twisted AFL. The mirabolic special cycle is a key ingredient for arithmetic induction on the geometric side.

\subsection{Arithmetic induction systems}

Let $F_0$ be a $p$-adic field. Consider tuples $(G, X, Y)$ of following data:
\begin{itemize}
    \item $G$ is a reductive group over $F_0$, $X$ is an affine variety over $F_0$ with $G$-action and $Y=(\BV, q)$ be a quadratic space over $F_0$ with an isometric linear algebraic action of $G$. 
\end{itemize}

For $y_0 \in Y$, denote its stabilizer in $G$ by $G_{y_0}$. We have $G(F_0) \backslash (X(F_0) \times G(F_0).y_0)= G_{y_0}(F_0) \backslash X(F_0)$.

\begin{definition}
A tuple $(G, X, Y)$ is good if for any $y_0 \in Y$ with $q(y_0) \in O_{F_0}^\times$, there exists a tuple $(G_{y_0}, X_{y_0}, Y_{y_0})$ and a rational isomorphism
\[
\mathfrak{c}_{y_0}: G_{y_0} \backslash X \to G_{y_0} \backslash (X_{y_0} \times Y_{y_0}).
\]
We call $\mathfrak{c}_{y_0}$ a relative Cayley map for $(G, X, Y, y_0)$.
\end{definition}

\begin{definition}\label{defn: arithmetic induction system}
For a good tuple $(G, X, Y)$, we say a locally constant function $F=F_{(G, X, Y)}$ on an open dense subset $U \subseteq X \times Y$ forms an arithmetic induction system for $(G, X, Y)$ if 
\begin{enumerate}
    \item $F$ is $G$-invariant.
    \item We have $F(x,y)=0$ unless $q(y) \in O_{F_0}$.  
    \item For $y_0 \in Y$, denote by $F_{y_0}=F(-, y_0) : X \to \BC$. If $(x, y_0) \in U$ with $q(y_0) \in O_{F_0}^\times$, then $\mathfrak{c}_{y_0}$ is defined on $U \times y_0$, and there exists a good tuple $(G_{y_0}, X_{y_0}, Y_{y_0})$ and a locally constant function $F_{(G_{y_0}, X_{y_0}, Y_{y_0})}$ on an open dense subset of $X_{y_0} \times Y_{y_0}$ such that
\[
F_{y_0} = F_{(G_{y_0}, X_{y_0}, Y_{y_0})} \circ \mathfrak{c}_{y_0}.
\]  
    \item The function  $F_{(G_{y_0}, X_{y_0}, Y_{y_0})}$ satisfies $(1)(2)(3)$ for the good tuple $(G_{y_0}, X_{y_0}, Y_{y_0})$.
\end{enumerate}
\end{definition}

Arithmetic induction systems for $(G, X, Y)$ form a $\BC$-vector space. 

\subsection{Relative Cayley maps and matchings}

We first consider the unitary side. Let $\BV=\BV^\flat \oplus Fu^{0}$ be a splitting of hermitian spaces where $(u^0, u^0)=1$. The element $u^0$ gives a splitting of the framing object $\BX=\BX^\flat \oplus \BE$, which induces embeddings
\[
i_{u_0}: \CN_{n-1} \cong \CZ(u^0) \hookrightarrow \CN_{n}, \quad X^\flat \mapsto X^\flat \oplus \CE,
\]
\[
i_{u_0, u_0^*}: \CN_{n-1}^\GL \cong \CZ(u^0, (u^0)^*) \hookrightarrow \CN_{n-1}^\GL, \quad X^\flat \mapsto X^\flat \oplus \CE.
\]

Any hermitian structure $A \in \Herm(\BV)$ could be written as 
\[
A=\begin{pmatrix}
a & b \\
c & d \\
\end{pmatrix}, \quad  a \in \Herm(\BV^\flat), b \in \BV^\flat,  c \in (\BV^\flat)^*, d \in F.
\]

\begin{definition}
Consider the $\U(V^\flat)$-equivariant rational relative Cayley map 
\begin{equation}
c_{\BV, \Herm}: \Herm(\BV) \to \Herm(\BV^\flat) \times \BV^\flat, \quad
A=\begin{pmatrix}
a & b \\
c & d \\
\end{pmatrix}
\mapsto (A^\flat= \frac{a}{1-d}, u^\flat=\frac{b}{1-d}).
\end{equation}
\end{definition}

We then consider the analytic side. Let $V_0=V_0^\flat \oplus Fe$ be a splitting of vector spaces over $F_0$. We have the induced splitting $V_0^*= (V_0^\flat)^* \oplus Fe^*$ where $(e^*, e)=1, e^*|_{V_0^\flat}=0$.
Any $\gamma \in \GL(V_0)$ could be written as \[
\gamma=\begin{pmatrix}
a & b \\
c & d \\
\end{pmatrix}, \quad  a \in \End(V_0^\flat), b \in \BV^\flat_0,  c \in (\BV^\flat_0)^*, d \in F_0.
\]

\begin{definition}
Consider the $\GL(V_0^\flat)$-equivariant rational relative Cayley map 
\begin{equation}
c_{V_0, \GL}: \GL(V_0) \to \GL(V_0^\flat) \times V_0 \times V_0^*, \quad
\gamma=\begin{pmatrix}
a & b \\
c & d \\
\end{pmatrix}
\mapsto (\gamma^\flat= \frac{a}{1-d}, u^\flat_1 =\frac{b}{1-d}, u^\flat_2=\frac{c}{1-d}).
\end{equation}
\end{definition}

\begin{proposition}
\begin{enumerate}
    \item If $(g, u^0)$ (resp. $(\gamma, e, e^*)$ ) is regualr semi-simple, then $c_{\BV, \Herm}(g)$ (resp. $c_{V_0, \GL}(\gamma)$ )is regular semi-simple. 
    \item If $(g, u^0)$ matches $(\gamma, e, e^*)$, then $c_{\BV, \Herm}(g)$ matches $c_{V_0, \GL}(\gamma)$. 
\end{enumerate}
\end{proposition}
\begin{proof}
As regular semi-simpleness is a condition on stabilizer of the group action, the first part follows. If $(g, u^0)$ matches $(\gamma, e, e^*)$, then by definition $(g, u^0, (u^0)^*)$ and $(\gamma, e, e^*)$ is conjugate in $\End(V) \times V \times V^*$. The relative Cayley map $c_V: \End(V) \to \End(V^\flat) \times V \times V^*$ given by $A \to (a, b/(1-d), c/(1-d))$ restricts to $c_{V_0, \GL}$ and $c_{\BV, \Herm}$. As $c_V$ is $\GL(V)$-equivariant, the result follows.
\end{proof}

\subsection{Arithmetic induction on geometric side}

\begin{proposition}\label{pullback of graphs}
Assume that $1-d \in O_F^\times$. Then for $g \in \GL(V)(F_0)$ such that $(g, u^0)$ is regular semi-simple. We have an equality of closed formal subscheme in $\CN_{n-1}^\GL$:
\begin{equation}
g \CN_n \cap_{\CN_n^\GL} \CN_{n-1}^\GL= g^\flat \CN_{n-1} \cap_{\CN_{n-1}^\GL} \CZ^\GL(u^\flat). 
\end{equation}
\end{proposition}
\begin{proof}
The group $\GL(\BV)$ acts on the set of principal polarizations on $\BX$, whose stabilizer at $\lambda_\BX$ can be identified with $\U(\BV)$. There $g \in \GL(\BV)/\U(\BV)$ can be regarded as a principal polarization $\lambda_g=g.\lambda_\BX$ on $\BX$. Then $g\CN_n \subseteq \CN_n^\GL$ is the lifting locus of the principal polarization $\lambda_g$. We study $(X, \iota, \rho) \in g \CN_n \cap_{\CN_n^\GL} \CN_{n-1}^\GL$. 

Firstly the new polarization $\lambda_g$ on $\BX=\BX^\flat \times \BE$ gives a polarization on $\BX^\flat$ which corresponds to $g^\flat$. Regard $\lambda_g$ as an endomorphism from $X^\flat \times  \BE$ to $(X^\flat)^\vee \times \BE^\vee$, which gives a lifting from $\BE$ to $(X^\flat)^\vee$ which gives a lifting of $u^\flat: \BE \to \BX$ via the lifting of the polarization $X \to X^\vee$ on $g \CN_n$. This shows $g \CN_n \cap_{\CN_n^\GL} \CN_{n-1}^\GL \subseteq  g^\flat \CN_{n-1} \cap_{\CN_{n-1}^\GL} \CZ^\GL(u^\flat).$ The other direction follows similarly.
\end{proof}

\begin{proposition}\label{geometric side: arithmetic induction}
Assume that $1-d \in O_F^\times$. Hence for $g \in \GL(V)(F_0)$ such that $(g, u^0)$ is regular semi-simple, we have 
\[
\Int^{\Herm, \BV}(g, u^0)= \Int^{\Herm, \BV^\flat}(c_{\BV, \Herm}(g)) \in \BQ.
\]
\end{proposition}
\begin{proof}
By Proposition \ref{prop: IntHerm=mirabolic int},
$$
\Int^{\Herm, \BV^\flat}(g^\flat, u^\flat)= \CN_{n-1} \cap^\BL_{\CN_{n-1}^\GL} g^\flat \CN_{n-1} \cap^\BL_{\CN^\GL_{n-1}} \CZ^\GL(u^\flat).
$$
The above intersection $g \CN_n \cap_{\CN_n^\GL} \CN_{n-1}^\GL= g^\flat \CN_{n-1} \cap_{\CN_{n-1}^\GL} \CZ^\GL(u^\flat)$ has correct codimension hence by \cite[Lemma B.2]{AFL-Wei2019} we have
\[
g \CN_n \cap_{\CN_n^\GL}^\BL \CN_{n-1}^\GL= g^\flat \CN_{n-1} \cap_{\CN_{n-1}^\GL}^\BL \CZ^\GL(u^\flat). 
\]
Intersecting both sides with $\CN_n$ inside $\CN_n^\GL$ and using identification $\CZ(u_0) \cong \CN_{n-1}$ and above Proposition \ref{pullback of graphs}, we obtain the result.
\end{proof}

\subsection{Arithmetic induction on analytic side}

Assume that the self-dual lattice $L=L^\flat \oplus O_Fe$ has an orthogonal decomposition where $(e,e)=1$, which induces the splitting $V_0=V_0^\flat \oplus F_0e$.

\begin{proposition}\label{analytic side: arithmetic induction}
Assume that $1-d \in O_F^\times$. Then for $\gamma \in \GL(V_0)(F_0)$ such that $(\gamma, e, e^*)$ is regular semi-simple, write $c_{V_0,\GL}(\gamma)=(\gamma^\flat, u_1^\flat, u_2^\flat)$. We have
\begin{equation}
\Orb(\gamma, e, e^*, \Phi'_{L'}, s)= \Orb(\gamma^\flat, u_1^\flat, u_2^\flat, \Phi'_{(L')^\flat},s)
\end{equation}
In particular, we have $\partial \Orb((\gamma,u'),  \Phi'_{L'}) = \partial \Orb(c_{V_0,\GL}(\gamma), \Phi'_{(L')^\flat}) \in \BQ \log q$.
\end{proposition}
\begin{proof}
This follows from the proof of \cite[Corollary 3.22.]{ZZhang2021} and the integral transitivity lemma \cite[Lemma 4.12]{ZZhang2021}. Note in our case, no twisting element is needed.
\end{proof}

\begin{remark}
We expect arithmetic induction systems to arise from orbital integrals in general. Consider the orbital integral of $f' \otimes \phi' \in \CS(X' \times Y')$. Choose a compact open subgroup $K_{H'} \leq H'$ such that $f'$ is $K_{H'}$-invariant. Fix $x_0 \in Y'$ with $(x_0, x_0)=1$. We know $H'$ acts transitively on $\{ x \in  Y' | (x,x)=1 \}$ whose stabilizer at $x_0$ is $H'^\flat$. If we assume the transitivity for the action of $K_{H'}$ on $\Omega:=\{ x \in Y' | (x,x)=1, \phi'(x) \not =0 \}$, then we have $\Orb_{H'}(g, x_0, (f' \otimes \phi') )=\vol(K_{H'}) \Orb_{H'^{\flat} }(g, f)$. If $f \in \CalS(X')$ goes to $f^\flat \times \phi^\flat \in \CalS(X'^\flat \times Y'^\flat)$ under the relative Cayley map, then generically we have  $\Orb_{H^\flat}(g,f)= \Orb_{H^\flat}(c_{V_0}(g), f^\flat \otimes \phi^\flat)$.
\end{remark}

\section{Twisted unitary Shimura varieties and special cycles}\label{section: global special cycles}

Starting with this section, we give the proof of the twisted AFL conjecture, i.e. Theorem \ref{Twisted AFL theorem} via global methods and Weil type relative trace formulas. Moreover, we introduce and study twisted CM cycles on unitary Shimura varieties. It turns out that for our local purpose on twisted AFL, we could work with punctured integral models away from infinitely many places and do not need to handle the constant term of global distributions. 

We switch to global set up. Let $F_0$ be a totally real number field, and $F/F_0$ be a CM quadratic extension. For the properness of our Shimura varieties, we always assume that $[F_0: \BQ]>1$. Let $E_0/F_0$ be a totally real quadratic extension. Then $E=E_0 \otimes_{F_0} F$ is a CM quadratic extension of $E_0$. Let $F'/F_0$ be the third subfield of $E$ fixed by $\sigma_{E_0/F_0} \otimes \sigma_{F/F_0}$, where $\sigma_{E_0/F_0}$ (resp. $\sigma_{F/F_0}$) the non-trivial Galois involution on $E_0/F_0$ (resp. $F/F_0$). We will study the embedding of algebraic groups 
$$H=\Res_{F_0/\BQ}\U(V) \to G=\Res_{E_0/\BQ}\U(V_{E_0})$$
over $\BQ$ for a suitable $F/F_0$-hermitian space $V$. In order to use moduli interpretations of Shimura varieties, we will work with the RSZ variant of unitary Shimura varieties \cite{RSZ-AGGP}.

\subsection{Shimura varieties for $V$ and $V_{E_0}$}\label{section: global set up of Shimura data}

Choose a CM type $\Phi \subseteq \Hom(F, \ov {\BQ})$ of $F$ with a distinguished element $\varphi_0 \in \Phi$. Let $V$ be a $F/F_0$-hermitian space of dimension $n \geq 1$ with signature $\{ (n-1,1)_{\varphi_0}, (n,0)_{\varphi \in \Phi -  \{ \varphi_0 \} }  \} $. Consider the following reductive groups over $\mathbb Q$:
\[
	Z^{\mathbb Q}:=\{ x \in  \Res_{F/\mathbb Q} \mathbb G_m | \, \Nm_{F/F_0} x \in \mathbb G_m \}, 
\]
\[
	H^{\mathbb Q} := \{g \in \Res_{F_0/\mathbb Q} \GU(V) | \, c(g) \in  \mathbb G_m \},
\]
\[
H=\Res_{F_0/\BQ} \U(V), \quad 	\wt{H}=Z^{\mathbb Q} \times H \cong Z^{\mathbb Q} \times_{\mathbb G_m} H^{\mathbb Q}.
\]
Here $c: \Res_{F_0/\mathbb Q} \GU(V)  \to  \Res_{F_0/\mathbb Q} \mathbb G_m$ is the similitude character. For neat compact open subgroups $K_{Z^{\mathbb Q}} \subseteq Z^{\mathbb Q}(\mathbb A_f), K \subseteq H(\mathbb A_f)$, consider the the \emph{RSZ unitary Shimura variety} with level $\wt{K}=K_{Z^{\mathbb Q}} \times K$ associated to the Shimura datum $(\wt{H}, \{ h_{\wt H} \} )$:
	\begin{equation}\label{eq: Shimura=Shimura x Shimura}
		\Sh_{\wt{K}}( \wt{H}, \{ h_{\wt{H}} \} )(\BC) \cong  \Sh_{K_{Z^{\BQ}}}( Z^{\BQ}, h_{\Phi})(\BC) \times \Sh_{K}( H, \{ h_{H} \})(\BC).
	\end{equation}
	
The canonical model of $\Sh_{K_{\wt{H}}}( \wt{H}, \{ h_{\wt{H}} \} )$
$$M:=M_{\wt{H}, \wt{K}} \to \Spec F^\RSZ$$
is a $n-1$ dimensional smooth \emph{projective} ($[F_0 : \BQ]>1$) variety over the reflex field $F^\RSZ$. 

Now we consider the CM type $\Phi_{E_0} \subseteq \Hom(E, \ov{\BQ})$ of $E$ induced by $\Phi$, i.e. $\varphi \in \Phi_{E_0}$ iff $\varphi|_{F} \in \Phi$. We denote the element of $\Phi_{E_0}$ over $\varphi_0 \in \Phi$ by $\{ \varphi_{0, E_0}, \varphi_{0, E_0}^c \}$, where $\varphi_{0, E_0}^c=\varphi_{0,E_0} \circ \sigma_{E/F}$. The base change $V_{E_0}$ is a $E/E_0$-hermtian space of dimension $n \geq 1$ with signature $\{ (n-1,1)_{\varphi_{0,E_0}}, (n-1,1)_{\varphi_{0, E_0}^c}, (n,0)_{\varphi \in \Phi_{E_0} - \{ \varphi_{0, E_0}, \varphi_{0, E_0}^c \} }  \} $. Similar to $V$, we consider the following reductive groups for $V_{E_0}$:
\[
	Z^{\mathbb Q, E_0}:=\{ x \in  \Res_{E/\mathbb Q} \mathbb G_m | \, \Nm_{E/E_0} x \in \mathbb G_m \}, 
\]
\[
	G^{\mathbb Q} := \{g \in \Res_{E_0/\mathbb Q} \GU(V_{E_0}) | \, c(g) \in  \mathbb G_m \},
\]
\[
G=\Res_{E_0/\BQ} \U(V_{E_0}), \quad 	\wt{G}=Z^{\mathbb Q, E_0} \times G \cong Z^{\mathbb Q, E_0} \times_{\mathbb G_m} G^{\mathbb Q}.
\]

For neat compact open subgroups $K_{Z^{\mathbb Q, E_0}} \subseteq Z^{\mathbb Q}(\mathbb A_f), K^{E_0} \subseteq G(\mathbb A_f)$, we have the canonical model of RSZ Shimura varieties for the Shimura datum $(\wt{G}, \{ h_{\wt{G}} \} )$:
$$
M^{E_0}=M_{\wt{G}, \wt{K^{E_0}}} \to \Spec E^\RSZ
$$
which is a $2n-1$ dimensional smooth \emph{projective} variety over the reflex field $E^\RSZ$. Base change gives a natural embedding of Shimura datum $(\wt{H}, \{ h_{\wt{H}} \}) \to (\wt{G}, \{h_{\wt{G}}\})$. Hence for compatible levels $\wt{K}$ and $\wt{K^{E_0}}$ (i.e. $\wt{K^{E_0}} \cap \wt{H}=\wt{K}$), we have a twisted diagonal embedding of Shimura varieties
\begin{equation}\label{twisted embedding: generic fiber}
M \times_{\Spec F^\RSZ} \Spec {E^\RSZ} \to M^{E_0}.    
\end{equation}

\subsection{Punctured PEL type integral models}\label{subsection: integral model}
For RSZ unitary Shimura varieties, we have moduli interpretations. 
Choose a hermitian lattice $L \subseteq V$ and a finite collection $\Delta$ of finite places for $F_0$ such that 
\begin{itemize}
    \item all $2$-adic places are in $\Delta$ and $F/F_0$ is unramified outside $\Delta$.
    \item $E_0/F_0$ is unramified outside $\Delta$.
    \item $L$ is self-dual away from $\Delta$.
    \item $\Phi$ is unramified away from $\Delta$.
\end{itemize} 
\begin{definition}\label{level for L}
Let $K(L)^{\Delta}$ (resp. $K^{E_0}(L)^\Delta$) be the stabilizer of $L$ (resp. $L_{O_{E_0}}$) in $ \U(V)(\mathbb A_{f}^{\Delta})$ (resp. $\U(V_{E_0})(\mathbb A_{f}^{\Delta})$) and $K_{Z^{\mathbb Q}}^{\Delta} $ (resp. $K_{Z^\BQ}^{E_0,\Delta}$) be the unique maximal compact open subgroup of 
	$Z^{\BQ}(\BA_{f}^{\Delta})$ (resp. $Z^{\BQ, E_0}(\BA_{f}^{\Delta})$). And we call
	\begin{equation}\label{level for L and Delta}
		{
			\wt{K}= K_{Z^\BQ, \Delta} \times K_{Z^{\mathbb Q}}^{\Delta} \times K_{\Delta} \times K(L)^{\Delta} \subseteq \wt{H}(\mathbb A_f).
       }
   \end{equation}
	a level structure for $L$ and $\Delta$, where $K_{Z^\BQ, \Delta}$ (resp. $K_{\Delta}$) is a sufficiently small compact open subgroup of $Z^\BQ(F_{0,\Delta})$ (resp. $\U(V)(F_{0, \Delta})$).  
 
We call  \begin{equation}\label{level for L and Delta and E0}
		{
			\wt{K^{E_0}}= K_{Z^\BQ, \Delta}^{E_0} \times K_{Z^{\mathbb Q}}^{E_0, \Delta} \times K_{\Delta}^{E_0} \times K^{E_0}(L)^{\Delta} \subseteq \wt{G}(\mathbb A_f).
       }
   \end{equation}
   a level structure for $L_{O_{E_0}}$ and $\Delta$, where $K_{Z^\BQ, \Delta}^{E_0}$ (resp. $K_{\Delta}^{E_0}$) is a sufficiently small enough compact open subgroup of $Z^\BQ(E_{0,\Delta})$ (resp. $\U(V)(E_{0, \Delta})$). 
\end{definition}	

Following \cite[Section 6.1]{RSZ-AGGP}, we introduce PEL type integral models over $O_{E^\RSZ}[\Delta^{-1}]$ for $M$ and $M^{E_0}$. For $?= F_0$ (resp. $?=E_0$), we set $O_{?,0}=O_{F_0}$ (resp. $O_{?,0}=O_{E_0}$), $O_{?}=O_{F}$ (resp. $O_{?}=O_{E}$). Now consider the functor
	$$
	\CM_0^{?} \to \Spec O_{E^\RSZ}[\Delta^{-1}]
	$$
	sending a locally noetherian scheme $S$ to the groupoid $\mathcal{M}_0^{?}(S)$ of tuples $(A_0, \iota_0, \lambda_0, \ov {\eta}_0)$ where
	\begin{itemize}
		\item $A_0$ is an abelian scheme over $S$ of dimension $[? : \mathbb Q]$.
		\item $\iota_0: O_?[\Delta^{-1}] \to \End(A_0)[\Delta^{-1}]$ is an $O_?[\Delta^{-1}]$-action satisfying the \emph{Kottwitz condition} of signature $\{(1,  0)_{\varphi \in \Phi^{?}} \}$:
		\begin{equation}
			\charpol (\iota_0(a) \mid \Lie A_0) =\prod_{\varphi \in \Phi^{?}} (T- \varphi(a)) \in \CO_S[T]. 
		\end{equation}
		\item $\lambda_0: A_0 \to A_0^\vee$ is an away-from-$\Delta$ principal polarization such that for all $a \in O_?[\Delta^{-1}]$ we have $ \lambda^{-1}_0 \circ \iota_0(a)^{\vee} \circ \lambda_0 = \iota_0(\overline{a})$.
		\item $\ov \eta_0$ is a $K_{Z^{\mathbb Q}, \Delta}^{?}$-level structure.
	\end{itemize}	

An isomorphism between two tuples is a quasi-isogeny preserving the polarization and $K_{Z^{\mathbb Q}, \Delta}$-level structure.  The functor $\mathcal{M}_0^{?}$ is representable, finite and \'etale.  We assume that $\CM_0^{?}$ is non-empty, which holds if $\Delta$ is large enough. The generic fiber $M_0^{?}$ of $\mathcal{M}_0^{?}$ is a disjoint union of copies of $\Sh_{K_{Z^\BQ}^{?}} (Z^{\mathbb Q, ?}, h_{\Phi^{?}})$, see \cite[Lemma 3.4]{RSZ-AGGP}. We work with one copy and still denote its \'etale integral model by $\CM_0^{?}$.

	\begin{definition}\label{defn: integral RSZ model}
		The integral RSZ Shimura variety with level $K^{?}$ for $L$ and $\Delta$ is the functor 
		$$
		\CM^{?} \to \Spec O_{E^\RSZ}[\Delta^{-1}]
		$$
		sending a locally noetherian scheme $S$ to the groupoid	of tuples 
		$(A_0, \iota_0, \lambda_0, \ov \eta_0,  A, \iota, \lambda, \overline{\eta}_{\Delta})$ where
		\begin{itemize}
			\item $(A_0, \iota_0, \lambda_0, \ov \eta_0)$ is an object of $\CM_0^{?}(S)$.
			\item $A$ is an abelian scheme over $S$ of dimension $n[?:\BQ]$.
			\item $\iota: O_?[\Delta^{-1}] \to \End(A)[\Delta^{-1}]$ is an $O_?[\Delta^{-1}]$-action satisfying the \emph{Kottwitz condition} of the signature of $V_{?}$.
			\item $\lambda : A \to A^{\vee}$ is a prime-to-$\Delta$ principal polarization such that for all $a \in O_?[\Delta^{-1}]$ we have $ \lambda^{-1} \circ \iota(a)^{\vee} \circ \lambda= \iota(\overline{a}). $
			\item $\overline{\eta}_{\Delta}$ is a $K_{\Delta}^{?}$-orbit of isometries of hermitian modules 
			\[
			\eta_{\Delta}: V_{\Delta}(A_0, A)= \prod_{v \in \Delta} V_{v}(A_0, A) \cong V_{?,\Delta}
			\]
			as smooth $F_{0, \Delta}$-sheaves on $S$. Here 
			$V_{v}(A_0, A)=\Hom_{F \otimes_{F_0} F_{0, v}} (V_vA_0, V_vA)$ is the Hom space between rational Tate modules of $A_0$ and $A$, with the hermitian form
			\[
			(x, y):= \lambda_0^{-1} \circ y^\vee \circ \lambda \circ x.
			\]
			\item For any finite place $v \not \in \Delta$ of $F_0$, we put the sign condition and Eisenstein condition at $v$ as in \cite[Section 4.1]{RSZ-AGGP}. By \cite[Section 4.1]{RSZ-AGGP} and \cite[Rem. 6.5 (i)]{RSZ-AGGP}, the Eisenstein condition is automatic when the places of $? \otimes_{F_0} F$ above $v$ are unramified over $p$, and the sign condition is automatic when $v$ is split in $F$. 
		\end{itemize}
		An isomorphism between two tuples is a pair of quasi-isogenies $(\phi_0, \phi): (A_0, A) \to (A_0', A')$ preserving polarizations and the $K_{Z^\BQ, \Delta}\times K_{\Delta}$-orbit of level structures.
	\end{definition}
	
	\begin{theorem} 
		The functor $\CM^?$ is representable by a separated scheme flat and of finite type over $\Spec O_E^\RSZ[\Delta^{-1}]$. Moreover, $\CM^?$ is smooth and projective over $\Spec O_{E^\RSZ}[\Delta^{-1}]$.
	\end{theorem}
\begin{proof}
   This follows from \cite[Theorem 6.2]{RSZ-Shimura}. Because of our assumption on $\Delta$, $E^\RSZ / F^\RSZ$ is also unramified away form $\Delta$, hence $\CM^?$ is still smooth over $\Spec O_{E^\RSZ}[\Delta^{-1}]$.
\end{proof}
 
	From now on, for simplicity we write $(A_0, A, \ov\eta)$ (resp. $A_{0}^{E_0}, A^{E_0}, \ov \eta^{E_0}$) for a $S$-point of $\CM$ (resp. $\CM^{E_0}$). Serre tensor construction gives a natural embedding 
 \begin{equation}
     \CM \to \CM^{E_0}, \, \, (A_0, A ,\ov \eta) \mapsto (A_0 \otimes_{O_{F}} O_{E}, A \otimes_{O_{F}} O_{E}, \ov \eta \otimes_{F} E ).
 \end{equation}
	
which recovers the embedding of Shimura varieties (\ref{twisted embedding: generic fiber}) over $E^\RSZ$.

\subsection{Arithmetic theta series of Kudla--Rapoport divisors} \label{section: arithmetic theta}

\subsubsection{On the unitary Shimura variety $M$}
For any $(A_0, A, \ov\eta) \in M(S)$, endow the $F$-vector space $\Hom^\circ_F(A_0,A)=\Hom_{O_F}(A_0,A)\otimes \BQ$ with the hermitian form
	\begin{equation}
		(u_1,u_2)
		:= \lambda_0^{-1}\circ u_2^\vee\circ \lambda \circ u_1 \in \End ^\circ_F(A_0)\simeq F.
	\end{equation}
	
	\begin{definition}\label{def:KR generic fiber}
		Consider non-zero $\xi\in F_{0}$ and $\mu\in  V(\BA_{0,f}) /K$.	The Kudla--Rapoport cycle $Z(\xi, \mu)$ is the functor sending any locally noetherian $E$-scheme $S$ to the groupoid of tuples
		$(A_0, A, \ov\eta, u)$ where
		\begin{itemize}
			\item  $(A_0, A, \ov \eta) \in M(S)$.
			\item  $u\in \Hom^\circ_F(A_0,A)$ such that $(u,u)=\xi$.
			\item $\ov\eta (u)$ is in the $K$-orbit $\mu$. 
		\end{itemize}
		An isomorphism between two tuples $(A_0, A, \ov\eta, u)$ and $(A_0', A', \ov\eta', u')$ is an isomorphism of tuples $(\phi_0, \phi): (A_0, A, \ov\eta) \simeq (A_0', A', \ov\eta') $ such that $\phi \circ u= u' \circ \phi_0$.
	\end{definition}
	
	By positivity of the Rosati involution, $Z(\xi, \mu) $ is empty unless $\xi \in F_{0,+}$ is totally positive. The natural forgetful morphism $i: Z(\xi, \mu) \to M$ is finite and unramified, and is \'etale locally a Cartier divisor. We view $Z(\xi, \mu)$ as elements in the Chow group (of $\BQ$-coefficients) $\Ch^1(M)$. 
 
    For a Schwartz function $\phi \in \CS(V(\BA_{0,f}))^K$ and $\xi \in F_{0,+}$, the $\phi$-averaged Kudla--Rapoport divisor is the finite summation 
	\begin{equation}\label{KR generic average}
		Z(\xi, \phi)\colon=\sum_{\mu\in V_\xi(\BA_{0,f})/K} \phi(\mu)\,Z(\xi, \mu) \in \Ch^1(M).
	\end{equation}
 
 Here $V_{\xi}=\{x \in V| (x,x)= \xi \}$ is the hyperboloid in $V$ of length $\xi$.  We put 
	\[ 
	Z(0, \phi) := -\phi(0) c_1(\omega)  \in \Ch^1(M)
	\]
	where $\omega$ is the automorphic line bundle (Hodge bundle) on $M$, which is a descent of the tautological line bundle on the hermitian symmetric domain $X$ of $M$ to $E$.

For $\xi \in F_0$, consider the weight $n$ Whittaker function (\ref{defn: Whittaker SL2}) on $\SL_2(F_{0, \infty})$:
	$$W^{(n)}_\xi (h_{\infty})=\prod_{v | \infty}  W^{(n)}_\xi (h_{v}).
	$$ Consider the Weil representation $\omega=\omega_{V,\psi}$ of $\SL_2(\mathbb A_{0,f})$ on $\CS(V(\BA_{0,f}))$ which commutes with the natural action of $\U(V)(\BA_{0,f})$.  For a Schwartz function $\phi \in \CS(V(\BA_{0,f}))^K$, we form \emph{the geometric generating series of Kudla--Rapoport divisors} on $M$ by
	\begin{equation}\label{eq: generating series generic fiber}
		Z(h, \phi) :=W^{(n)}_0(h_{\infty}) Z(0, \omega(h_f) \phi ) + \sum_{\xi \in F_{0, +}} W^{(n)}_\xi (h_{\infty})  Z(\xi, \omega(h_f) \phi )\in \Ch^1(M),
	\end{equation}
	for any $h=(h_\infty, h_f) \in \SL_2(\BA_{F_0})=\SL_2(F_{0, \infty}) \times \SL_2(\BA_{0, f})$. For any $h_f \in \SL_2(\mathbb A_{0,f})$, we have 
	\begin{equation}\label{eq: dual relation on generic fiber}
		Z(h, \omega(h_f)\phi)=Z(hh_f, \phi). 
	\end{equation}	

 The modularity of geometric theta series of divisors is known by the thesis work of Yifeng Liu \cite{Liu-Thesis}.
	    
	\begin{proposition}\cite[Thm. 8.1]{AFL-Wei2019}\label{modularity over generic fiber} 
		For any $\phi \in \CS(V(\BA_{0,f}))^K$, $Z(h, \phi)$ is a weight $n$ holomorphic automorphic form valued in $\Ch^1(M_{\wt H, \wt K})$. More precisely, we have
		$$Z(h, \phi) \in \CA_{\rm hol}(\SL_2(\BA_{F_0}), K_0, n)_{\ov \BQ} \otimes_{\ov \BQ} \Ch^1(M)_{\ov \BQ}
  $$
		where $K_0 \subseteq \SL_2(\BA_{0,f})$ is any compact open subgroup that fixes $\phi$ under the Weil representation.
	\end{proposition}

\subsubsection{On integral model $\CM$}

To do arithmetic intersection, we consider the arithmetic theta series on $\CM$. Now $K$ is a level structure for $L$ and $\Delta$ (Definition (\ref{level for L})).
	
\begin{definition}\label{defn: global int KR}
		Let $\xi\in F_{0,+}$ and $\mu_\Delta \in V(F_{0,\Delta})/K_{\Delta}$.  \emph{The Kudla--Rapoport cycle}  $\CZ(\xi, \mu_{\Delta}) \to \CM$ is the functor sending any locally noetherian scheme $S$ over $O_{E^\RSZ}[1/\Delta]$ to the groupoid of tuples $(A_0, A, \ov\eta, u)$ where
		\begin{itemize}
			\item  $(A_0, A, \ov\eta) \in\CM(S)$;
			\item  $u \in \Hom_{O_F}(A_0,A)[1/\Delta]$ with $(u,u)=\xi$.
			\item $\ov\eta (u)$ is in the $K_{\Delta}$-orbit $\mu_{\Delta}$. 
		\end{itemize}
	\end{definition} 

\begin{definition}
A Schwartz function $\phi \in \CS(V(\BA_f))^{K}$ is standard for $L$, if $\phi=\phi_{\Delta} \otimes \phi^{\Delta}$ where where $\phi^\Delta=1_{\wh{L}^\Delta}$ agrees with the indicator of $L$ in $\CS(V(\BA_f^\Delta))^K$ away from $\Delta$. 
\end{definition}	
Consider $\phi \in \CS(V(\BA_{0,f}))^{K}$ that is standard for $L$. For $\xi \in F_{0,+}$, the $\phi$-averaged Kudla--Rapoport divisor is the finite summation
	\begin{equation}\label{def:KR int}
		\CZ(\xi, \phi)\colon=\sum_{\mu_{\Delta} \in V_\xi(F_{0,\Delta})/K_{\Delta}} \phi_\Delta(\mu_{\Delta})\,\CZ(\xi, \mu_{\Delta}),
	\end{equation}
	viewed as elements in the Chow group $\Ch^1(\CM)$.  
	
Let $\mu$ be an archimedean place of $E^\RSZ$. we have two kinds of Green functions \cite[Section 3]{AFL-Wei2019} for the Kudla--Rapoport divisor $Z(\xi,\phi)_\BC$ on $M \otimes_{E^\RSZ, \nu}\BC$:
	\begin{itemize}
		\item The \emph{Kudla Green function} $\CG^{\bf K}(\xi,h_{\infty},\phi)$ with a variable $h_\infty \in \SL_2(F_{0, \infty})$, which occurs in the analytic side naturally via computations of archimedean orbital integrals. 
		\item The \emph{automorphic Green function} \cite{JBGreen} $\CG^{\bf B}(\xi, h_\infty, \phi)=\CG^{\bf B}(\xi,\phi)$ for $\xi \in F_{0, +}$ which is admissible by \cite[Cor. 5.16]{JBGreen}. 
	\end{itemize} 

We consider \emph{arithmetic Kudla--Rapoport divisors} ($?=\bf K, \bf B$)
 \begin{equation}
 	\wh \CZ^{?} (\xi, h_\infty, \phi)= (\CZ (\xi, \phi), ( \mathcal{G}^{?}_{\nu}(\xi, h_\nu, \phi))_{\nu} ) \in \wh{\Ch}^1(\CM).  
 \end{equation}
 
\subsection{Derived twisted CM cycles}

We use the twisted diagonal embedding to define derived twisted CM cycles on $\CM$.

\begin{definition}\label{defn: Hecke correspondence away from Delta}
For $\mu_{\Delta}^{E_0} \in K_\Delta^{E_0} \backslash \U(V)(E_{0, \Delta}) / K_\Delta^{E_0}$, the \emph{Hecke correspondence away from $\Delta$} 
 $$
	\Hk_{\mu_{\Delta}^{E_0}} \to \CM^{E_0} \times \CM^{E_0}
$$ 
is the functor sending any locally noetherian $O_{E^\RSZ}[\Delta^{-1}]$-scheme $S$ 
to the groupoid of tuples $(A_0^{E_0}, A_1^{E_0}, \ov \eta_{A_1^{E_0}}, A_2^{E_0}, \ov \eta_{A_2^{E_0}}, \varphi)$ where
		
		\begin{enumerate}
			\item $(A_0^{E_0}, A_1^{E_0}, \ov \eta_{A_1^{E_0}}), \, (A_0, A_2^{E_0}, \ov \eta_{A_2^{E_0}}) \in \CM^{E_0}(S)$.
			\item $\varphi \in \Hom_{O_E}(A_1^{E_0}, A_2^{E_0})[1/\Delta]$ is a prime-to-$\Delta$ quasi-isogeny such that $\varphi^*\lambda_2=\lambda_1$.
			\item  We have $\eta_{A_2^{E_0}} \varphi \eta_{A_1^{E_0}}^{-1} \in \mu_{\Delta}^{E_0} \subseteq \U(V)(E_{0,\Delta}) $ for some $\eta_{A_1^{E_0}} \in \ov \eta_{A_1^{E_0}}$ and $\eta_{A_2^{E_0}} \in \ov \eta_{A_2^{E_0}}.$
		\end{enumerate}
\end{definition}

Recall above $\eta_{A_i^{E_0}}$ ($i=1,2$) is an isometry of hermitian modules $\eta_{A_i^{E_0}}: V_\Delta(A_0, A) \cong V_{E_0, \Delta}$. So naturally we have $\eta_{A_2^{E_0}} \varphi \eta_{A_1^{E-0}}^{-1} \in \U(V)(E_{0,\Delta})$. Recall we have an embedding $H=\Res_{F_0/\BQ} \U(V) \leq G=\Res_{E_0/\BQ} \U(V_{E_0})$. We now pullback these Hecke correspondences from $G \times G$ to $H \times H$ via Serre tensor constructions and define the \emph{twisted CM cycles}.

Recall over $E^{\RSZ}$, we may attach a complex point $(A_0^{E_0}, A^{E_0}, \ov{\eta}) \in M^{E_0}(\BC)$ a $F/F_0$ hermitian space via Betti cohomology
\[
V(A_0^{E_0}, A^{E_0}):= \Hom_{F}(H_1(A_0, \BQ), H_1(A, \BQ))
\]
The existence of $\ov \eta$ implies that we have an isomorphism $V(A_0^{E_0},A^{E_0}) \cong V_{E_0}$ as hermitian spaces.

\begin{definition}
For an element $\ov{\alpha} \in \U(V)(F_0) \backslash \U(V)(E_0) / \U(V)(F_0)$, the twisted CM cycle for $\alpha$ and $\mu_\Delta^{E_0}$ over $E^\RSZ$
\begin{equation}
CM^{E_0}(\ov{\alpha}, \mu^{E_0}_\Delta) \to M \times M   
\end{equation}
is the functor sending any locally noetherian $E^\RSZ$-scheme $S$ 
to the groupoid of tuples $(A_0, A_1, \ov \eta_{A_1}, A_2, \ov \eta_{A_2}, \varphi)$ where 
\begin{enumerate}
    \item $(A_0, A_1, \ov \eta_{A_1}), \, (A_0, A_2, \ov \eta_{A_2}) \in M(S)$.
    \item $\varphi \in \Hom_{O_E}(A_1 \otimes_{O_F} O_{E}, A_2 \otimes_{O_F} O_{E} )[1/\Delta]$ is a prime-to-$\Delta$ quasi-isogeny such that $\varphi^*\lambda_2=\lambda_1$.
	\item The induced map $\varphi_*: H_1(A_1^{E_0}, \BQ) \cong H_1(A_2^{E_0}, \BQ)$ on each complex point transforms to $\alpha \in \U(V_{E_0})$ for some $\alpha \in \ov{\alpha}$, under an identification $V(A_0^{E_0}, A^{E_0}_i) \cong V_{E_0}$ of hermitian spaces that induces $\eta_{A_i}: V_{\Delta}(A_0^{E_0}, A^{E_0}) \cong V_{E_0, \Delta}$ for some $\eta_{A_i^{E_0}} \in \ov{\eta}_{A_i^{E_0}}$.
\end{enumerate}
\end{definition}

\begin{remark}
If $E_0=F_0 \times F_0$, then by \cite[(2.11.))]{AFL-Wei2019} the categorical quotient $\U(V) \backslash \U(V_{E_0}) / \U(V)$ over $F_0$, could be identified via characteristic polynomials with an affine subscheme $\mathcal{A}_n \subseteq \Res_{F/F_0} F[T]_{\deg =n}$ consist of conjugate self-reciprocal monic polynomials.
\end{remark}

By considering the relative position of $\varphi$ over generic fibers, we have a decomposition of Hecke cycles into \emph{twisted CM cycles}
\begin{equation}
(\CM \times \CM ) \cap_{\CM^{E_0} \times \CM^{E_0}} \Hk_{\mu^\Delta} = \coprod_{\ov{\alpha}} \mathcal{CM}^{E_0}(\ov{\alpha}, \mu^{E_0}_\Delta)
\end{equation}
where $\ov{\alpha}$ runs over $\U(V)(F_0) \backslash \U(V)(E_0) / \U(V)(F_0)$, such that $\mathcal{CM}^{E_0}(\ov{\alpha}, \mu^{E_0}_\Delta)_{E^\RSZ}=CM^{E_0}(\ov{\alpha}, \mu^{E_0}_\Delta)$.

\begin{definition}
We form the derived intersection in the K-group of $\CM^{E_0}$ and decomposition as above:
\[
(\CM \times \CM ) \cap^\BL_{\CM^{E_0} \times \CM^{E_0}} \Hk_{\mu^\Delta} = \coprod_{\alpha} {}^\BL \mathcal{CM}^{E_0}(\ov \alpha, \mu^{E_0}_\Delta),
\]
and ${}^\BL \mathcal{CM}^{E_0}(\ov \alpha, \mu^{E_0}_\Delta) \in F_1'K'_{0,\CM^{E_0}(\ov \alpha, \mu^{E_0}_\Delta)}(\CM \times \CM)$ is called the derived twisted CM cycle for $(\ov \alpha, \mu^{E_0}_\Delta)$.
\end{definition}

\begin{definition}\label{defn: irreducible in (H, G, H)}
An element $\ov{\alpha} \in \U(V)(F_0) \backslash \U(V)(E_0) / \U(V)(F_0)$ is called irreducible over $F$, if its image in $ \GL(V) \backslash \GL(V_{E_0}) / \GL(V) $ contains some element $\alpha^{E_0}$ whose characteristic polynomial over $E$ is irreducible, i.e. $E[\alpha^{E_0}]$ is a field.
\end{definition}

By the theory of complex multiplication (for $A^{E_0}$), if $\ov{\alpha}$ is irreducible, then $^{\BL}\mathcal{CM}^{E_0}(\alpha, \mu^{E_0}_\Delta)$ is $1$-dimensional away from finitely many primes with $0$-dimensional generic fiber $CM^{E_0}(\alpha, \mu^{E_0}_\Delta)$ over $E^\RSZ$. See also \cite[Prop. 7.9]{AFL-Wei2019}.

\begin{remark}
It is possible to give moduli interpretations of integral twisted CM cycles $\mathcal{CM}^{E_0}(\ov \alpha, \mu^{E_0}_\Delta)$  without using Hecke correspondences, by requiring the induced map 
$\eta_{A_2^{E_0}} \varphi_* \eta_{A_1^{E_0}}^{-1}: V_\Delta \cong V_\Delta(A_0^{E_0}, A^{E_0}_1) \to V_\Delta(A_0^{E_0}, A^{E_0}_2) \cong V_\Delta $ agrees with the image of $\alpha$ in $\U(V)(F_{0,\Delta}) \backslash \U(V)(E_{0,\Delta}) / H(F_{0,\Delta})$. This agrees with the above definition because we put the sign condition on integral model $\CM$ using sign invariants as in \cite[Section 4.1]{RSZ-Shimura}. For our purpose, knowing the information on generic fibers already gives the desired local-global decomposition of global intersection numbers into local orbital integrals and intersection numbers.
\end{remark}

Let $\mathcal C_1(\CM)$ be the $\BQ$-vector space spanned by $1$-cycles on $\CM$, and $\mathcal C_1(\CM)_{\sim}$ be the quotient of $\mathcal C_1(\CM)$ by the subspace generated by $1$-cycles on $\CM$ that are contained in a closed fiber and rationally trivial within that fiber.
	
\begin{definition}[derived twisted CM cycles] \label{defn: twisted Derived CM}
	For a Schwartz function $f=f_\Delta \otimes 1_{(K^{E_0})^\Delta} \in \CS(G(\BA_f))^{K^{E_0} \times K^{E_0}}$, the derived twisted CM cycle for $f$ is the finite summation
		\begin{equation}
			{}^{\BL}\mathcal{CM}^{E_0}(\alpha, f)= \sum_{\mu_{\Delta}^{E_0} \in K^{E_0}_{\Delta} \backslash \U(V)(E_{0,\Delta}) / K^{E_0}_{\Delta} } f_\Delta (\mu_\Delta^{E_0}) 
			^{\BL}\mathcal{CM}^{E_0}(\alpha, \mu^{E_0}_\Delta) \in \CC_1(\CM)_{\sim}.
		\end{equation}
Here we regard $^{\BL}\mathcal{CM}^{E_0}(\alpha, \mu^{E_0}_\Delta) \in \CC_1(\CM)_{\sim}$ via the the projection map $\mathrm{pr}_1: \CM \times \CM \to \CM$ to the first factor.
\end{definition}

\section{Local--global decomposition formulas and transfers}\label{section: local-global}			

In this section, we globalize two sides of the twisted AFL. Let $\xi \in F_{0}$. For $\phi=\phi_\Delta \otimes 1_{\wh{L}^\Delta} \in \CS(V(\BA_f))^{K}$ and $f=f_\Delta \otimes 1_{K^E_0}  \in \CS(\U(V_{E_0})(\BA_f))^{K_{E_0} \times K_{E_0}}$, we consider arithmetic intersection number
$\Int(\xi, f \otimes \phi)$ and study its local-global decomposition formulas following the argument of \cite[Theorem 3.9]{AFL-Invent} \cite[Theorem 8.15]{RSZ-AGGP}. 

\subsection{Geometric generating functions}

We globalize the geometric side in Section \ref{section: TAFL}.  Consider the quotient $\BQ$-vector space 
$$
\BR_\Delta:=\BR \, / \sum_{v|\ell, v \in \Delta}{ \BQ \log \ell}.
$$
Consider the $\Delta$-punctured arithmetic intersection pairing between $\BQ$-vector spaces:
	\begin{equation}\label{eq: Truncated Arithmetic int pairing}
		(\cdot,\cdot): \wh\Ch^1(\CM)\times \CC_{1}(\CM)  \to \BR_\Delta.
	\end{equation}   
  
Recall the volume factor $\tau(Z^\BQ)=\#Z^\BQ(\BQ) \backslash Z^\BQ(\BA_f)$.

 \begin{definition}[$\xi$-part of global arithmetic intersection numbers] \label{defn: global arithmetic intersection numbers}
For $\xi \in F_0$ and $\Phi=f \otimes \phi \in \CS((\U(V_{E_0}) \times V)(\BA_{0,f}))^{(K_{E_0} \times K_{E_0}) \times K}$ as above, we define \emph{global arithmetic intersection numbers}
	\begin{equation}
		\Int^{?}(\xi, \Phi) := \frac
		{1}{\tau(Z^{\BQ}) [E^\RSZ:F]}
		(\wh \CZ^{?} (\xi, \phi), {}^{\BL}\mathcal{CM}^{E_0}(\alpha, f)) \in \BR_\Delta. 
	\end{equation}
We set $\Int(\xi, \Phi)=\Int^{\bf B}(\xi, \Phi)$. 
 \end{definition}

If $\xi \not =0$, we may lift $\Int(\xi, \Phi)$ to $\BR$ using the physical divisor $\wh \CZ^{\bf B} (\xi, \phi)$. The global intersection decomposes into semi-global intersection numbers at each local place $v$ of $F_0$:
\begin{equation}\label{global int=sum of semi-global int at each v}
\Int(\xi, \Phi)=\sum_{v} \Int_v(\xi, \Phi).
\end{equation}
We now connect the local intersection pairing $\Int_v(\xi, \Phi)$ via basic uniformization to our local set up for the twisted AFL.

\subsection{Basic uniformization of special cycles at inert places}

Assume that $v_0 \not \in \Delta$ is inert in $F$. 
Let $V^{(v_0)}$ be the nearby $F/F_0$-hermitian space of $V$ at $v_0$, which is positively definite at all $v | \infty$ and locally isomorphic to $V$ at all finite places $v \not = v_0$ of $F_0$.

Let $\nu$ be a finite place of $E^\RSZ$ with $\nu|_{F}=w$, $\nu|_{F_0}=v_0$. By our assumption that $\Phi$ is unramifeid over $p$, we have $\breve{E}^\RSZ_\nu=\breve{F_{w}}=\breve{F_{0,v}}$.

Denote by $\CM^{\wedge}$ be the formal completion of $\CM_{O_{\breve{E}^\RSZ_\nu}}$ along the basic locus of the geometric special fiber $\CM \otimes \BF_{\nu}$.  The formal scheme $\CM^{\wedge}$ is formally smooth and \emph{formally of finite type} over $O_{\breve{E}^\RSZ_\nu}$. Consider the (relative) unitary Rapoport--Zink space 
$$
 \CN_n \to \Spf O_{\breve{F}}$$
(associated to $L_{v_0}$) as in Section \ref{section: TAFL}. 

 Consider reductive groups $H^{(v_0)}:=\U(V^{(v_0)})$ over $F_0$ and $\wt{H}^{(v_0)}:=Z^{\BQ} \times \Res_{F_0/\BQ} \U(V^{(v_0)})$ over $\BQ$.  We have the following basic uniformization theorem.

	\begin{proposition}\label{prop: basic uniformization RSZ}
		There is a natural isomorphism of formal schemes over $O_{\breve{E}^\RSZ_\nu}=O_{\breve{F_0}}$:
		\begin{equation}\label{eq: basic basic unif}
			\Theta: \, \, \CM^{\wedge} \stackrel{\sim}{\rightarrow}  \wt{H}^{(v_0)}(\BQ) \backslash ( \CN^{\prime} \times \wt{H} (\BA_{f}^{p}) / \wt{K}^{p} ), 
		\end{equation}  
		where 
		\[ \mathcal{N}^{\prime} \simeq (Z^{\BQ} \left(\mathbb{Q}_{p}\right) / K_{Z^{\BQ}, p} ) \times \mathcal{N}_n \times \prod_{v|p, v \not = v_0} \U(V)(F_{0, v}) / K_{v}. \]
		There is a natural projection map
		\begin{equation}\label{projection map from RSZ to unitary}
			\CM^{\wedge}  
			\to Z^{\BQ}(\BQ)  \backslash (Z^{\BQ}(\BA_{0, f}) / K_{Z^\BQ} ) 
		\end{equation}  
		with fibers isomorphic to
		\[ \Theta_0: \CM^{\wedge}_0 \stackrel{\sim}{\rightarrow}  H^{(v_0)}(F_0) \backslash [ \mathcal{N}_n  \times H (\mathbb{A}_{0, f}^{v_0}) / K^{v_0} ]. \] 
	\end{proposition}
 \begin{proof}
See \cite[Thm 8.15]{RSZ-AGGP}. 
 \end{proof}
	
	Hence we denote by $\CM^{\wedge}_0$ any fiber of $\mathcal{M}^{\wedge}$ under the projection map (\ref{projection map from RSZ to unitary}). 
 
 Note that $v_0$ is split or inert in $E_0$. Consider the Rapoport--Zink space $\CN_n^{E_{0,v_0}}:=\CN^\GL_n$ (resp. $\CN_n \times \CN_n$ ) when $v_0$ is inert in $E_0$ (resp. $v_0$ splits in $E_0$).  Similarly, we have basic uniformization of the formal completion $\CM^{E_0, \wedge}$ of $\CM^{E_0}$ along its basic locus.

\begin{proposition}\label{prop: basic uniformization twisted RSZ}
There is a natural isomorphism of formal schemes over $O_{\breve{E}^\RSZ_\nu}=O_{\breve{F_0}}$:
    \begin{equation}\label{eq: basic basic unif E0}
			\Theta^{E_0}: \, \, \CM^{E_0, \wedge} \stackrel{\sim}{\rightarrow}  \wt{G}^{(v_0)}(\BQ) \backslash ( \CN^{E_{0, v_0} \prime} \times \wt{G} (\BA_{f}^{p}) / \wt{K^{E_0}}^{p} ), 
		\end{equation}  
  where 
  \[ \CN^{E_{0, v_0} \prime} \simeq (Z^{\BQ, E_0}(\BQ_p) / K_{Z^{\BQ, E_0}, p}) \times \CN^{E_{0, v_0}}_n \times \prod_{v|p, v \not = v_0} \U(V)(E_{0, v}) / K_{v}^{E_0}. \]
	There is a natural projection map
		\begin{equation}\label{projection map from RSZ to twisted unitary}
			\CM^{E_0, \wedge}  
			\to Z^{\BQ, E_0}(\BQ)  \backslash (Z^{\BQ, E_0}(\BA_{0, f}) / K_{Z^\BQ, E_0} ) 
		\end{equation}  
		with fibers isomorphic to
		\[ \Theta_0^{E_0}: \CM^{E_0, \wedge}_0 \stackrel{\sim}{\rightarrow}  H^{(v_0)}(F_0) \backslash [ \CN_n^{E_{0,v_0}}  \times H (\mathbb{A}_{0, f}^{v_0}) / K^{v_0} ]. \] 
\end{proposition}
\begin{proof}
This is a special case of basic uniformization theorem for general Hodge type Shimura varieties, see \cite[Theorem 4.10.6]{PappasRapoport-GlobalShimura}. The same construction of uniformization maps in \cite[Thm 8.15]{RSZ-AGGP} applies here and gives the desired isomorphism.
\end{proof}
		
Set $V_{\xi}:=\{ u \in  V | (u, u)=\xi \}$ be the subset of $V$ consists of elements of norm $\xi$.

\begin{proposition}\cite[Prop 7.4.]{AFL-Wei2019}
Assume that $\xi \not =0$. The restriction of special divisor $\CZ(\xi,\phi)$ on each fiber $\CM^\wedge_0$ of the projection is identified with
\begin{equation}
\sum_{(g,u) \in \U(V)(F_0) \backslash ( \U(V)(\BA^{v}_f)/K^v \times V_{\xi} )  } \phi^{v}(g^{-1}u) [\CZ(u) \times 1_{gK^v}] .
\end{equation}
\end{proposition}

Recall the Hecke correspondence (Definition \ref{defn: Hecke correspondence away from Delta})
 $$
	\Hk_{\mu_{\Delta}^{E_0}} \to \CM^{E_0} \times \CM^{E_0}.
$$ 
Let $\Hk_{\mu_{\Delta}^{E_0}}^{\wedge}$ be the locally noetherian formal scheme over $O_{\breve{F_0}}$ obtained as pullback of $\Hk_{\mu_{\Delta}^{E_0}}$ on $\CM^{E_0, \wedge} \times \CM^{E_0, \wedge}$. Consider reductive groups $G^{(v_0)}=\Res_{E_0/F_0}\U(V^{(v_0)}_{E_0})$ over $F_0$ and $\wt{G^{(v_0)}}= Z^{\BQ, E_0} \times \Res_{E_0/\BQ} \U(V^{(v_0)}_{E_0})$ over $\BQ$.

Denote by $\Hk_{\mu_{\Delta}^{E_0},0}^{\wedge}$ any fiber of the natural projection map sending $(A_0, A^{E_0}, \ov \eta^{E_0})$ to $A_0$
$$
\Hk_{\mu_{\Delta}^{E_0}}^{\wedge} \to Z^{E_0, \BQ}(\BQ) \backslash Z^{E_0, \BQ}(\BA_f)/ K_{Z^{E_0, \BQ}}.
$$ 
Consider the discrete set
	\[
	\Hk_{\mu^{E_0}_{\Delta},0}^{(v_0)}= \{ (g_1,g_2) \in G(\BA_f^{v_0})/K^{E_0, v_0} \times G(\BA_f^{v_0})/K^{E_0, v_0} | g_1^{-1}g_2 \in K^{E_0}\mu_{\Delta}^{E_0} K^{E_0} \}. 
	\]

We recall the compatibility of Hecke correspondences and basic uniformization.
 
	\begin{theorem}
	    We have a natural isomorphism of  formal schemes
		\[ 
		\Hk_{\mu^{E_0}_{\Delta}, 0}^{\wedge} \cong G^{(v_0)}(F_0) \backslash [ \CN^{E_0, v_0}_n \times \Hk_{\mu_{\Delta}^{E_0},0}^{(v_0)}]
		\]
		compatible with the projection under the isomorphism $\Theta$ (c.f. (\ref{eq: basic basic unif})). 
		
	\end{theorem}
	\begin{proof}
	Although we work over $\CM^{E_0}$, the same proof of \cite[Prop. 7.16]{AFL-Wei2019} (on $\CM$) applies.
	\end{proof}

Assume $v_0$ is inert in $E_0$. For $\delta \in G^{(v_0)}(F_{0,v})/H^{(v_0)}(F_{0,v})$, consider the twisted fixed cycles of $\delta$ on $\CN_n$ (see Definition \ref{twisted fixed cycles})
$$
{}^\BL\CN^{\Herm, v_0}(\delta)= \delta \CN_n \cap^{\BL}_{\CN_n^{E_0,v_0}} \CN_n \in K_0'(\CN_n).    
$$
For $(\delta, h) \in B^{(v_0)}(F_0) \times H(\BA_{0,f}^{v_0}) / K^{v_0} $, consider the twisted CM cycle
	\begin{equation}
		{}^{\BL} \mathcal{CN}^{E_0} (\delta, h)_{K^{E_0, v_0}} := [{}^{\BL} \CN^{\Herm}(\delta) \times 1_{hK^{E_0, v_0}}] \rightarrow [ \CN_n \times H(\BA_f^{v_0})/K^{ v_0}].
	\end{equation}

 The summation 
	$$
 \sum_{(\delta', h')}  {}^\BL \mathcal{CN}^{E_0} (\delta', h')_{K^{v_0}} $$ over the $H^{(v_0)}(F_0)$-orbit of $(\delta, h)$ in $G^{(v_0)}(F_0)/ H^{(v_0)}(F_0) \times H(\BA_{0,f}^{v_0}) / K^{v_0}$ descends to a cycle $[ {}^\BL \mathcal{CM} (\delta, h) ]_{K^{v_0}}$ on the quotient $\CM^{\wedge}$ by above Theorem \ref{prop: basic uniformization twisted RSZ}. As in the proof of \cite[Thm. 3.9]{AFL-Invent}, we have

\begin{proposition}
For a Schwartz function $f=1_{K^{E_0}} \otimes f_{\Delta} \in \CS(\U(V_{E_0})(\BA_{0, f}), K^{E_0})$, the restriction of the basic uniformization ${}^\BL \mathcal{CM}^{E_0}(\alpha, f)^{\wedge}$ of twisted CM cycles to $\CM^{E_0, \wedge}_0$ is the sum
		\begin{equation}\label{basic derived Hk}
			\sum_{(\delta, h) \in  H^{(v_0)}(F_0) \backslash   B^{(v_0)}(F_0) \times H(\BA_{0,f}^{v_0})/K^{v_0} } f^{v_0}(h. \delta) [{}^\BL\mathcal{CM} (\delta, h) ]_{K_G^{v_0}}.
		\end{equation}
\end{proposition}

\subsection{Analytic generating functions}

We now globalize the analytic side in Section \ref{section: twisted orbits, FL}. Let $F'/F_0$ be the third quadratic extension inside $E$ fixed by $\sigma_{E_0/F_0} \otimes \sigma_{F/F_0}$. Consider 
\[
H'_1= \GL_{n,F_0} \rightarrow G'=\GL_{n,F'} \leftarrow H'_{20}=\GL_{n,F_0}.
\]

Consider the symmetric space
\begin{equation}
	S_{n, F'/F_0}=\{ \gamma \in \GL_{n,F'} | \gamma \sigma_{F'/F_0} (\gamma) =id \} 
\end{equation}
By Hilbert 90, we have $S_{n, F'/F_0} \cong G' / H'_{20}$. Consider the natural quadratic space $V' =F_0^n \times (F_0^n)^*$.

Let $\eta=\eta_{F/F_0}: \BA_{F_0} \to \BC^\times$ be the quadratic character associated to $F/F_0$ by global class field theory. Similar to the local set up, consider $\eta$-twisted orbital integrals  for the diagonal action of $h \in H'_1(F_0)=\GL_n(F_0)$
on the product $(\gamma,u_1,u_2) \in S_{n,F'/F_0} \times  V'$  by
\begin{equation}
	h.(\gamma,u_1,u_2)=(h^{-1}\gamma h,h^{-1}u_1, u_2h).
\end{equation}

\begin{proposition}\label{matching of global orbits: Un to twisted Un}.
There is a natural identification of categorical quotients
\[
[H_1' \backslash G' / H_{20}']=[\U(V) \backslash \U(V_{E_0}) / \U(V) ]
\]
compatible with determinant maps to $[ \BG_{m,F}^{Nm_{F/F_0}=1} \backslash \BG_{m,E}^{Nm_{E/E_0}=1} ].$ In particular, we get a matching of regular semi-simple orbits 
\[
[H_1' \backslash G' / H_{20}']_\rs=\coprod_{V} [\U(V) \backslash \U(V_{E_0}) / \U(V) ]_\rs
\]
where $V$ ranges over $n$-dimensional $F/F_0$-hermitian spaces.
\end{proposition}
\begin{proof}
This is a globalization of Proposition \ref{matching of orbits: Un to GLn}. The Lie algebra version is straightforward. As this is a geometric statement, we may base change to $E_0$ and the result follows from the Jacquet Rallis case (computed via Igusa's criterion) \cite[Lemma 3.1]{zhang2014fourier}. 
\end{proof}

Let $\ov{\alpha}$ be an element in $\U(V)(F_0) \backslash \U(V)(E_0) / \U(V)(F_0) $. Assume that $\ov{\alpha}$ is irreducible. Consider the subscheme $S_{n, F'/F_0}(\alpha) \subseteq S_{n, F'/F_0}$ of elements corresponds to $\alpha$ under above identification of categorical quotients.  Then $S_{n, F'/F_0}(\alpha)(F_0)$ consists of exactly one $\GL_n(F_0)$-orbit.

 Consider a decomposable function $$
\Phi'=\otimes_{v} \Phi'_v \in \CS((S_{n, F'/F_0} \times V')(\BA_{F_0}))
$$ such that for every infinite place $v|\infty$ of $F_0$, $\Phi'_v$ is the chosen partial Gaussian test function (relative to $\alpha$)  introduced in \cite[Section 12.4]{AFL-Wei2019}.

For $(\gamma, u') \in (S_{n, F'/F_0}(\alpha)\times V')(F_0)$ that is regular semisimple (equivalently $u' \not =0$ as $\alpha$ is irreducible), define the global $\eta$-orbital integral ($v$ runs over all places of $F_0$):
\begin{equation}
	\Orb((\gamma,u'), \Phi',s)= \prod_{v} \Orb((\gamma,u'), \Phi'_v,s)
\end{equation}
where $\eta$ is the quadratic character associated to $F/F_0$ by class field theory.

Denote by $[(S_{n, F'/F_0}(\alpha) \times V')(F_0)]$ (resp. $[(S_{n, F'/F_0}(\alpha) \times V')(F_0)]_\rs$) the set of (resp. regular semisimple) $\GL_n(F_0)$-orbits in $(S_{n, F'/F_0}(\alpha) \times V')(F_0)$.    For $h \in \SL_2(\BA_{F_0})$ and $s \in \BC$, consider the  regularized integral as in \cite[Section 11.4]{AFL-Wei2019}:
\begin{equation}
	\BJ(h,\Phi',s)=\int_{g \in [\GL_n(F_0) \backslash \GL_n(\BA_{F_0}) } \left(\sum_{(\gamma, u')\in (S_{n,F'/F}(\alpha)\times V')(F_0)}\omega(h)\Phi'(g^{-1}. (\gamma, u'))\right)   |g|^s\eta(g)  dg.
\end{equation}

By \cite[Thm. 12.14]{AFL-Wei2019}, we see that $\BJ(h,\Phi',s)$ is a smooth function of $(h, s)$, entire in $s \in \BC$ and is left invariant under $h \in \SL_2(F_0)$ by Poisson summation formula. We have a decomposition:
\begin{equation}
	\BJ(h,\Phi',s)= \BJ(h,\Phi',s)_{0}+\sum_{(\gamma,u')\in [(S_{n, F'/F_0}(\alpha) \times V')(F_0)]_\rs} \Orb((\gamma,u'),\omega(h)\Phi',s),
\end{equation}
where $\BJ(h,\Phi',s)_{0}$ is the term over the two regular nilpotent orbits as in \cite[Section 12.6]{AFL-Wei2019}.

For $\xi \in F_0^\times$, let $V'_\xi$ be the $F_0$-subscheme of $V'$ defined by $\{(u_1, u_2) \in V'| u_2(u_1)= \xi \}$. The $\xi$-th Fourier coefficient  of $\BJ(\cdot,\Phi',s)$ is equal to
\begin{align}
	\sum_{(\gamma,u')\in [(S_{n, F'/F_0}(\alpha) \times V'_\xi)(F_0)]_\rs} \Orb((\gamma,u'),\omega(h)\Phi',s).
\end{align}
Then we introduce
\begin{equation}
	\partial \BJ(h,\Phi') :=\frac{d}{ds}\Big|_{s=0}   \BJ(h,\Phi',s).
\end{equation}
\begin{equation}
	\del((\gamma,u'),\Phi'_v) := \frac{d}{ds}\Big|_{s=0}  \Orb((\gamma,u'),\Phi'_v,s).
\end{equation}

By Leibniz's rule, we have the standard decomposition:
\begin{equation}
	\partial \BJ(h,\Phi') = \partial \BJ(h,\Phi')_{0} + \sum_{v}    \partial \BJ_v(h, \Phi'),
\end{equation}
where $v$ runs over all places of $F_0$ and
\begin{equation}\label{def partial BJ a,v }
	\partial \BJ_v(h, \Phi') :=\sum_{(\gamma,u')\in [(S_{n, F'/F_0}(\alpha) \times V')(F_0)]_\rs}  \del((\gamma,u'), \omega(h)\Phi'_v)\cdot  \Orb((\gamma,u'), \omega(h)\Phi'^{v}).
\end{equation}
And the nilpotent term $\partial \BJ(h,\Phi')_{0}$ is part of the $0$-th Fourier coefficient of $\partial \BJ(h,\Phi')$. We do not need its precise formula in this paper. Consider the Fourier expansion
\begin{equation}\label{eq: Del BJ}
	\partial \BJ_v(h, \Phi') = \sum_{\xi \in F_0} \partial \BJ_v(\xi, h, \Phi'),
\end{equation}
\[
\partial \BJ_v(\xi, h, \Phi')= \sum_{(\gamma,u')\in [(S_{n, F'/F_0}(\alpha) \times V'_\xi)(F_0)]_\rs}  \del((\gamma,u'), \omega(h)\Phi'_v)\cdot  \Orb((\gamma,u'), \omega(h)\Phi'^{v}).
\]

\begin{proposition}\label{modularity of del J}
The function $\partial \BJ(h,\Phi')$ is a smooth function on $h \in \SL_2(\BA_{F_0})$, left invariant under $\SL_2(F_0)$, and of parallel weight $n$. If $\Phi'$ is $K_0$-invariant under the Weil representation, then  $\partial \BJ(h,\Phi')$ is right $K_0$-invariant.      
\end{proposition}
\begin{proof}
This follows from similar properties of $\BJ(h,\Phi',s)$ from above discussion.
\end{proof}

\begin{proposition}[local-global decomposition: analytic side]\label{local-global decomp analytic}
Assume that $\xi \not =0$. Then
\begin{enumerate}
    \item If $v \in \Inert(F/F_0)$, then 
\[
\partial \BJ_v(\xi, 1, \Phi')= \sum_{(\gamma,u')\in [(S_{n, F'/F_0}(\alpha) \times V'_\xi)(F_0)]_\rs}  \del((\gamma,u'), \Phi'_v)\cdot  \Orb((\gamma,u'), \Phi'^{v}).
\]
   \item If $v$ is an archimedean place of $F_0$, then 
\[
\partial \BJ_v(\xi, 1, \Phi')= \sum_{(\gamma,u')\in [(S_{n, F'/F_0}(\alpha) \times V'_\xi)(F_0)]_\rs}  \del((\gamma,u'), \Phi'_v)\cdot  \Orb((\gamma,u'), \Phi'^{v}).
\]
\end{enumerate}
\end{proposition}
\begin{proof}
This follows from above discussion by taking $h=id$.
\end{proof}

We will see in next subsection that if $\Phi'$ is a good transfer of $\Phi$ then $\partial \BJ_v(\xi, \Phi')=0$ for $v \in \Split(F/F_0)$.

\subsection{Local-global comparison and transfers}

\begin{definition}\label{unitary orbital integral}
Let $v$ be a finite place $v$ of $F_0$. For Schwartz function $f_v \in \CS(\U(V_{E_0,v_0})$ and $\phi_v \in \CS(V_{v_0})$, the orbital integral of $\Phi_v=f_v \otimes \phi_v$ at a regular semi-simple orbit of $(g,u) \in [\U(V) \backslash (\U(V_{E_0})(\alpha)/ \U(V) \times V_\xi) ](F_{0,v_0})$ is the following integral:
\begin{equation}\label{unitary orbit integral}
\Orb((g,u), \Phi_v)=\int_{h \in \U(V)(F_{0,v})} (\int_{h_2 \in \U(V)(F_{0,v})} f_v(h^{-1}gh_2)  dh_2) \phi_v(h^{-1}u) d h,
\end{equation}
where the Haar measure on $\U(V)(F_{0,v})$ is normalized such that a hyperspecial maximal compact open subgroup has volume $1$.
\end{definition}

Let $\Phi=f \otimes \phi \in\CS( \U(V_{E_0}) \times V)(\BA_{0,f}))$ be a stanford decomposable function with level $K_{E_0}$ and $K$. In particular, $\Phi^{\Delta}=1_{\U(L_{O_{E_0}})^\Delta} \times 1_{\wt{L}^\Delta}$.

\begin{proposition}[local-global decomposition: geometric side]\label{prop: global int=local int}
Assume that $\xi \not =0$. Let $v|p$ be a place of $F_0$. Then
\begin{enumerate}
    \item If $v \in \Split(F/F_0)$, then $\Int_v(\xi, \Phi)=0$.
    \item If $v \in \Inert(F/F_0)$, then 
   \[
\Int_v(\xi, \Phi)= \log q_v^2 \sum_{(g, u) \in [\U(V) \backslash (\U(V_{E_0})(\alpha)/ \U(V) \times V_\xi)](F_0)]} \Int^{\Herm, v_0}(g,u) \Orb((g,u), \Phi^{v_0}).
   \] 
\end{enumerate}
\end{proposition}
\begin{proof}
We follow the proof of \cite[Theorem 9.4]{AFL-Wei2019}. The first part is the same: if $\Int_v(\xi, \Phi)$ is non-zero, then for any intersection point $(A_0, A, \ov \eta) \in \mathcal{CM}^{E_0}(\ov{\alpha}, \mu^{E_0}_\Delta) \cap \CZ(\xi, \mu_\Delta)(\BF_p^{alg})$, we could construct isogeny from $A_0^n$ to $A$ to show that is in the supersingular locus of the mod $v$ fiber of $\CM$, and $v$ is inert or ramified in $F$.

The second part follows by combining the above uniformization (\ref{eq: basic basic unif E0}). See also \cite[Theorem 3.9]{AFL-Invent} \cite[Theorem 8.15]{RSZ-AGGP}. 
\end{proof}

To compare the above local-global decompositions, we now define the notion of transfers (relative to $\alpha$).

\begin{definition} \label{Partial transfers}
We say functions $\Phi=f \otimes \phi \in\CS( \U(V_{E_0}) \times V)(\BA_{0,f}))$ and $\Phi'=f' \otimes \phi' \in \CS((S_{n, F'/F_0} \times V')(\BA_{F_0}))$ are \emph{partial transfers} (relative to $\alpha$), if
\begin{itemize}
    \item for any archimedean place $v$ of $F_0$, $\Phi'_v$ is a Gaussian test function (relative to $\alpha$).
    \item for any finite place $v$ of $F_0$, $\Phi_v'$ and $\Phi_v$ are required to have same orbital integrals (\ref{unitary orbit integral}) for matching regular semi-simple orbits (from Proposition \ref{matching of global orbits: Un to twisted Un})
    $$
    (\gamma, u_1, u_2) \in [\GL_n(F_{0,v}) \backslash S_{n, F'_v/F_v}(\alpha) \times V')(F_{0,v})]_\rs  
    $$
   and 
   $$
   (g, u) \in [\U(V) \backslash \U(V_{E_0})(\alpha)/ \U(V) \times V')(F_{0,v})]_\rs.  
   $$
\end{itemize} 
\end{definition}	

Let $\Phi=f \otimes \phi$ and $\Phi'=f' \otimes \phi'$ be a pair of partial transfers (relative to $\alpha$).

\begin{proposition}
If $v \in \Split(F/F_0)$, then $\partial \BJ_v(\xi, \Phi)=0$.
\end{proposition}
\begin{proof}
This follows from the same proof of \cite[Lemma 14.2]{AFL-Wei2019}.
\end{proof}

\begin{proposition}\label{local-global decomp geometric}
Assume that $\xi \not =0$ and $v | \infty$ is an archimedean place of $F_0$. Then
\[
2 \partial \BJ_{v}(\xi, \Phi') + \Int^{\mathbf K}_v(\xi, \Phi)=0.
\]
\end{proposition}
\begin{proof}
As $E_0/F_0$ is split at archimedean places, this is the same as \cite[Theorem 10.1]{AFL-Wei2019} by our choice of partial Gaussian test function $\Phi'_\infty$.
\end{proof}

\begin{corollary}\label{cor: local AFL implies semi-global}
Assume that the twisted AFL conjecture \ref{Twisted AFL conjecture} for all $(g,u) \in O_{V^{(v)}}(\alpha, \xi)(F_0)$ at $v$ holds true, and $\Phi^v, \Phi'^v$ are transfers on $O_{V^{(v)}}(\alpha, \xi)(F_0)$, then we have 
\[
\Int_v(\xi, \Phi) + 2 \partial \BJ_v (\xi, \Phi')=0.
\]
\end{corollary}
\begin{proof}
This follows from comparison of right hand sides of Proposition \ref{local-global decomp analytic} and Proposition \ref{prop: global int=local int}. The archimedean comparison is done by Proposition \ref{local-global decomp geometric}.
\end{proof}

\begin{proposition}\label{prop: shrink}
Assume that $\Orb((g,u), 1_{K_w} \otimes 1_{L_w}) \not =0$ for all finite places $w \nmid v\Delta$ (which always holds after enlarging $\Delta$). Fix a compact subset $\Omega_v$ of $[\U(V^{v}_{E_0})/ \U(V^{v}) \times V^{v}](F_{0,v})$ (in $v$-adic topology). Then after possibly shrinking $K_\Delta$, there exists choice for $\phi^{v}=\phi_\Delta \otimes 1_{\widehat{L}^{v\Delta}} \in \CS(V(\BA_f^v))$ and $f=f_\Delta \otimes 1_{K^{E_0,v\Delta}} \in \CS( K^{E_0,v} \backslash \U(V)(\BA_f^{v}) / K^{E_0, v})$ such that for $(g_1,u_1) \in [\U(V) \backslash \U(V_{E_0})(\alpha)/ \U(V) \times V'_\xi)(F_{0})]_\rs$ the following are equivalent
\begin{enumerate}
    \item We have that $\Orb((g_1, u_1), \Phi^v) \not =0$ where $\Phi^v=f^v \otimes \phi^v$, and that $(g_1. u_1)$ is in the $\U(V^{v})(F_{0,v})$-orbit of $\Omega_v$.
    \item The pair $(g_1,u_1)$ is in the $\U(V^{v})(F_{0,v})$-orbit of $(g,u)$.
\end{enumerate}

\end{proposition}
\begin{proof}
This follows from the same proof of \cite[Lemma 13.7]{AFL-Wei2019}.
\end{proof}

From Proposition \ref{prop: shrink}, we obtain the converse of the Corollary \ref{cor: local AFL implies semi-global}. Choose $\Omega_v$ such that non-zero terms in right-hand side of Proposition \ref{local-global decomp analytic} and Proposition \ref{local-global decomp geometric} always intersects with $\Omega_v$.

\begin{corollary} \label{cor: local equals to global}
The twisted AFL conjecture \ref{Twisted AFL conjecture} for $(g,u)$ holds if and only if ($\xi_0=(u,u)$)
\[
2\partial_v \BJ (\xi_0, \Phi') + \Int_v(\xi_0, \Phi)=0.
\]
\end{corollary}

\section{The proof of twisted AFL}\label{section: final proof}

In this section, we finish the proof of twisted AFL (Theorem \ref{Twisted AFL theorem}) using Corollary \ref{cor: local equals to global}.

\subsection{Globalization}

Let $F_{v_0} / F_{0, v_0}$ be an unramified quadratic extension of $p$-adic local fields with residue fields $\BF_{q^2}/\BF_{q}$. Now we give a proof of twisted AFL Conjecture \ref{Twisted AFL conjecture} for $F_{v_0}/F_{0,v_0}$ and a regular semi-simple orbit
\[
(g_0, u_0) \in [\U(\BV_{v_0})(F_{v_0}) \backslash (\GL(\BV_{v_0})(F_{v_0}) / \U(\BV_{v_0})(F_{v_0})  \times \BV_{v_0})]_\rs.
\]
	
We can find a CM quadratic extension $F/F_0$ of a totally real field $F_0$ with a distinguished embedding $\varphi_0: F_0 \hookrightarrow \mathbb R$, a totally real quadratic extension $E_0/F_0$, a $n$-dimensional $F/F_0$-hermitian space $V$, a CM type $\Phi$ of $F$ such that
	\begin{itemize}
		\item $\Phi$ is unramified at $p$.
		\item There exists a place $v_0|p$ of $F_0$ inert in $F$ and split in $E_0$ such that the completion of $F/F_0$ at $v_0$ is exactly $F_{v_0} / F_{0, v_0}$. All places $v \not = v_0$ of $F_0$ above $p$ is split in $F$. 
		\item $V$ is of signature $(n-1, 1)$ at $\varphi_0$, and $(n, 0)$ at all other real places of $F_0$. 
		\item The localization $V \otimes_{F} F_{v_0} $ is isomorphic to the hermitian space $V_{v_0}$ we start with. The space $V_p=\prod_{v|p} V_v$ contains a lattice $L_p = \prod_{v| p} L_v$ such that $L_v$ is self-dual. 
	\end{itemize}
		
	Let $V^{(v_0)}$ be the nearby $F/F_0$ hermitian space of $V$ at $v_0$ so the localization of $V^{(v_0)}$ at $v_0$ is isomorphic to $V^{(v_0)}_{v_0} = \BV_{v_0}$. As in \cite[Section 10.1]{AFL-Wei2019}, by local constancy of intersection numbers \cite{AFL-constant} (whose argument is quite general and applies to our Rapoport--Zink space $\CN_n^\GL$), we can find a pair $(g, u) \in (\U(V^{(v_0)}) \backslash \U(V^{(v_0)}_{E_0}) / \U(V^{(v_0)}) \times V^{v_0} )(F_0)_\rs $ such that
	
	\begin{itemize}
		\item $(g, u)$ is $v_0$-closely enough to $(g_0, u_0)$ such that two sides of Conjecture \ref{Twisted AFL conjecture} are the same for $(g, u)$ and $(g_0, u_0)$. The characteristic polynomial $\alpha$ of $g$ is irreducible over $F$.
		\item The norm $\xi_0= (u,  u)_{V^{v_0}}  \in F_0$ is non-zero.
		\item For all places $v \not = v_0$ of $F_0$ above $p$, the orbital integral for the $\U(V_v)$-orbit of $(g, u)$
		\[
		\Orb((g, u), 1_{\U(L_v)} \times 1_{L_v} ) \not =0. 
		\]
	\end{itemize}
	
	Choose a finite collection $\Delta$ of finite places of $F_0$ such that $p \nmid \Delta$, i.e. all places $v$ of $F_0$ above $p$ are not in $\Delta$.  
	
	We can enlarge $\Delta$ (keeping $p \nmid \Delta$) to assume that $\Delta$ and $L$ are in the global setup (c.f. \ref{subsection: integral model}) and $\Phi$ is unramified at all $\ell \not \in \Delta$. In particular, $E_0/F_0$ is unramified from $\Delta$ (but there are infinitely many places of $F_0$ inert in $E_0$).
 Consider a level structure $\wt{K}$ for $L$ and $\Delta$. 

\subsection{The final proof}
We choose $\Phi=f \otimes \phi$ and $\Phi'=f' \otimes \phi'$ be a pair of partial transfers (relative to $\alpha$) satisfying the requirement of Proposition \ref{prop: shrink}.

As in the proof of AFL in \cite{AFL-Wei2019}, the twisted AFL identity for $(g,u)$ will follow from 
\begin{equation}\label{desired global identity for TAFL}
2 \partial \BJ(\xi_0, \Phi') + \Int(\xi_0, \Phi) + \Int^{\mathbf{K-B}}_\infty(\xi_0, \Phi)=0 \, \text{in $\mathbb R_\Delta$.} 
\end{equation}
Consider the difference
$$
\textbf{Diff}(h) =2 \partial \BJ(h, \Phi') + \Int(h, \Phi) + \Int^{\mathbf{K-B}}_\infty(h, \Phi),
$$
with Fourier coefficients ($\xi \not =0$)
$$
\textbf{Diff}(\xi) =2 \partial \BJ(\xi, \Phi') + \Int(\xi, \Phi) + \Int^{\mathbf{K-B}}_\infty(\xi, \Phi).
$$

As the modularity of arithmetic theta series is currently unknown for general $F_0 \not =\BQ$, we finish the proof via a modification following \cite[Section 10.6]{AFL-JEMS}. We may enlarge $\Delta$ and assume there exists a characteristic polynomial $\alpha_m$ such that $\alpha_m$ is of maximal order away from $\Delta$. By the maximal order case of twisted AFL (Proposition \ref{TAFL: maximal order}) and \ref{cor: local AFL implies semi-global}, we have
\begin{equation}\label{maximal order diff, xi_0}
2 \partial \BJ_{\alpha_m}(\xi_0, \Phi') + \Int_{\alpha_m}(\xi_0, \Phi) + \Int^{\mathbf{K-B}}_{\alpha_m, \infty}(\xi_0, \Phi)=0, 
\end{equation}
Here the notions $\partial \BJ_{\alpha_m}, \Int_{\alpha_m}$ means we replace $\alpha$ by $\alpha_m$ in the definition of $\partial \BJ, \Int_{\alpha_m}$.

Then the twisted AFL identity for $(g,u)$ will follow from the equation 
\begin{equation}
2 \partial \BJ(\xi_0, \Phi')^\circ + \Int(\xi_0, \Phi)^\circ + \Int^{\mathbf{K-B}}_\infty(\xi_0, \Phi)^\circ=0 \, \text{in $\mathbb R_\Delta$.} 
\end{equation}
where each term is the difference between corresponding terms in equation (\ref{desired global identity for TAFL}) for $\alpha$ and $\alpha_m$. By the geometric modularity result \cite{Liu-Thesis} we know $\Int(h, \Phi)^\circ$ is a homolophirc automorphic forms on $h \in \SL_2(\BA_{F_0})$. 

Togethere with Proposition \ref{modularity of del J} and \cite[Proposition 10.5]{AFL-JEMS}, we see that the difference 
$$
\textbf{Diff}^\circ(h) =2 \partial \BJ(h, \Phi')^\circ + \Int(h, \Phi)^\circ + \Int^{\mathbf{K-B}}_\infty(h, \Phi)^\circ
$$
is a homolophirc automorphic forms on $h \in \SL_2(\BA_{F_0})$. 

By arithmetic inductions on twisted AFL and local-global decomposition formulas (Proposition \ref{local-global decomp analytic} and Proposition \ref{local-global decomp geometric}), similar to \cite[Lemma 13.6]{AFL-Wei2019} we have
\begin{equation}
\textbf{Diff}(\xi)=0, \, \text{if $(\xi, \Delta)=1$.}
\end{equation}
Similarly, 
\begin{equation}\label{maximal order diff}
2 \partial \BJ_{\alpha_m}(\xi, \Phi') + \Int_{\alpha_m}(\xi, \Phi) + \Int^{\mathbf{K-B}}_{\alpha_m, \infty}(\xi, \Phi)=0, \text{if $(\xi, \Delta)=1$.}
\end{equation}
Hence
\begin{equation}
\textbf{Diff}^\circ(\xi)=\textbf{Diff}(\xi)=0, \, \text{if $(\xi, \Delta)=1$.}
\end{equation}
As $\textbf{Diff}^\circ$ is a homolophirc automorphic forms on $h \in \SL_2(\BA_{F_0})$, by [Lemma 13.6]\cite{AFL-Wei2019} we deduce
\begin{equation}
\textbf{Diff}^\circ(h)=0 \, \text{in $\mathbb R_\Delta$.} 
\end{equation}
In particular, $\textbf{Diff}^\circ(\xi_0)=0$ hence $\textbf{Diff}(\xi_0)=0$ from (\ref{maximal order diff, xi_0}).
This finishes the proof of twisted AFL identity by Corollary \ref{cor: local equals to global}.

\section{Arithmetic twisted Gan-Gross-Prasad conjectures for Asai L-functions} \label{section: TAGGP}

In this section, we formulate arithmetic analogs of the twisted Gan-Gross-Prasad conjecture \cite{twistedGGP}\cite{Wang-TGGP} and discuss a relative trace formula approach following \cite[Section 5]{Liu-FJcycles}, where the twisted arithmetic fundamental lemma plays a key role. The related cycles for global height pairings are produced by Hecke actions on twisted unitary Shimura varieties and certain doubling/magic divisors (extending Kudla's geometric theta series) on self-product of unitary Shimura varieties. The construction is certain ``motivic dedoubling'' of Kudla-Rapoport divisors.

We return to the set up of Section \ref{section: global set up of Shimura data}. So $F/F_0$ (resp. $E/F_0$) is a CM (resp. totally real) quadratic extension of number fields. And $V$ is a $F/F_0$-hermitian space of dimension $n \geq 1$ with signature $\{ (n-1,1)_{\varphi_0}, (n,0)_{\varphi \in \Phi -  \{ \varphi_0 \} }  \} $. 
The embedding of algebraic groups 
$$H=\Res_{F_0/\BQ}\U(V) \to G=\Res_{E_0/\BQ}\U(V_{E_0})$$
induces the embedding of (RSZ) unitary Shimura varieties over $E^\RSZ$
$$
X_K:=M_{\wt{H},\wt{K}} \to Y_K:=M_{\wt{G}, \wt{K_{E_0}}}
$$
for compatible levels $K \leq H(\BA_f), K_{E_0} \leq G(\BA_f)$. In particular, $\dim Y_K= 2 \dim X_K=2(n-1)$.

Let $\mu$ be a conjugate symplectic automorphic character of $\BA_F^\times/F^\times$ of weight one with associated number field $M_\mu$ \cite[Definition 4.1, 4.3]{Liu-FJcycles}. Let $\mu^c$ be the complex conjugate of $\mu$. Fix a CM data $D_\mu = (A_\mu, i_\mu, \lambda_\mu, r_\mu)$ \cite[Definition 4.5]{Liu-FJcycles}. Here $A_\mu$ is an abelian variety over $F$ with a CM structure $i_\mu: M_\mu \to \End_F(A)_\BQ$, a polarization $\lambda_\mu: A_\mu \to A^\vee_\mu$ and an isomorphism $r_\mu: M_\mu \otimes_\BQ F \cong H_1^{dR}(A/F)$ of $M_\mu \otimes_\BQ F$-modules. In particular, $\dim A_\mu = [M_\mu: \BQ]/2$.

Let $A_K$ be the Albanese variety of $X_K$, and $A_\infty=\lim_K A_K$ (resp. $Y_\infty=\lim_K Y_K$)be the inverse limit of $A_K$ over all levels $K$.Consider the $M_\mu[\U(V)(\BA_f)]$-module
$
\Omega(\mu):=\Hom_F(A_\infty, A_\mu)_\BQ.
$

\begin{theorem}
(\cite[Theorem 1.1]{Liu-FJcycles})
There is an isomorphism 
\[
\Omega(\mu) \otimes_{M_\mu} \BC = \oplus_{\epsilon} \oplus_{\chi} \omega(\epsilon, \mu, \chi) 
\]
Here $\omega(\epsilon, \mu, \chi)$ is the finite part of the Weil representation of $\U(V)$ determined by the index $(\epsilon, \mu, \chi)$ \cite[Appendix D]{Liu-FJcycles} appearing in the (twisted) Fourier-Jacobi GGP conjecture \cite{twistedGGP}.
\end{theorem}

Consider the twisted diagonal embedding 
\[
\Delta^{E_0 \times F_0}: X_K \to Y_K \times X_K.
\]
Recall the specific Hecke correspondence $T_\mu^{\can}$ on $A_\mu$ \cite[Definition 4.9]{Liu-FJcycles}. Let $\alpha_K : X_K \to A_K$ be the Albanese morphism sending the zero-dimensional cycle $D^{n-1}$ to zero, where $D_K$ is the canonical extension of the Hodge divisor on $X_K$. 

\begin{definition}
For $f \in S(K_{E_0} \backslash G(\BA_f) / K_{E_0} )$ and $\phi \in \Hom_F(A_K, A_\mu)$, let $T^f_K$ be the Hecke correspondence associated to $f$ on $Y_K$. The twisted Fourier-Jacobi cycle for $(f, \phi)$ is the cohomological trivial cycle on $Y_K \times A_\mu$
\begin{equation}
\FJ(f, \phi)_K:=  |\pi_0^{\mathrm{geo}}(X_K)| (T^f_K \otimes T_\mu^\can)^* (\id_{Y_K} \times (\phi \circ \alpha_K) )_* (\Delta^{E_0 \times F_0}X_K) \in \CH^{n-1+[M_\mu:\BQ]/2}(Y_K \times A_\mu)^0_\BC.
\end{equation}
as an element in the cohomological trivial Chow group $\CH^{n-1+[M_\mu:\BQ]/2}(Y_K \times A_\mu)^0_\BC$ (with $\BC$-coefficients). Here we add the number $|\pi_0^{\mathrm{geo}}(X_K)|$ of geometric connected components of $X_K$ to make sure that $\FJ(f, \phi)_K$ is compatible with change of levels.
\end{definition}
Let $\Omega(\mu,\epsilon) := \oplus_{\chi} \omega(\epsilon, \mu, \chi)$ and $\Hom_E(A_K, A_\mu,\epsilon):=\Hom_E(A_K, A_\mu) \cap \Omega(\mu, \epsilon)$. 

Let $\pi$ be a cuspidal automorphic representation of $\U(V_{E_0})(\BA_{E_0})$. Let $\Pi$ be the automorphic representation of $\GL_n(\BA_{E})$ as the base change of $\pi$. Assume that $\Pi$ is relevant \cite{Liu-FJcycles}. Following \cite{twistedGGP}, we have the twisted Asai L-function
\[
L(s, \Pi, \As_{E/F} \otimes \mu)
\]
as a memorphic functions on $s \in \BC$ with central point $s=1/2$. 

By the same argument of \cite[Theorem 1.3]{Liu-FJcycles}, assuming the conjecture on the injectivity of the $\ell$-adic Abel–Jacobi map, the assignment $(f, \phi) \to \FJ(f, \phi)_K$ induces a complex linear map
\[
\FJ_\epsilon: \pi_f \otimes_\BC \Omega(\mu, \epsilon) \to \Hom_{\BC[\U(V_{E_0})(\BA_f)]} (\pi_f, \CH^{n-1+[M_\mu:\BQ]/2}(Y_\infty \times A_\mu)^0_\BC) 
\]

We now formulate the (unrefined) arithmetic twisted Gan-Gross-Prasad conjecture (similar to \cite[Conjecture 1.4]{Liu-FJcycles}).

\begin{conjecture}[Arithmetic Twisted Gan-Gross-Prasad Conjecture] \label{conj: ATGGP}
Let $F/F_0$ (resp. $E/F_0$) be a CM (resp. totally real) quadratic extension of number fields. Let $\mu$ be a conjugate symplectic automorphic character of $\BA_F^\times/F^\times$ of weight one. Let $\pi$ be a cuspidal automorphic representation of $\U(V_{E_0})(\BA_{E_0})$. Let $\Pi$ be the automorphic representation of $\GL_n(\BA_{E})$ as the base change of $\pi$. Assume that $\Pi_\infty$ is relevant \cite[Definition 1.2]{Liu-FJcycles}. Then the following statements are equivalent.
\begin{enumerate}
    \item The twisted arithmetic functional $\FJ_\epsilon \not =0$.
    \item We have $L'(1/2, \Pi, As_{E/F} \otimes \mu) \not =0$ and 
   $ \Hom_{\BC[\U(V)(\BA_f)]}(\pi_f \otimes \Omega(\mu, \epsilon), \BC) \not =0.$
\end{enumerate}
\end{conjecture}

There is a relative trace formula approach towards refinements of above arithmetic twisted Gan-Gross-Prasad conjecture. Following \cite{Liu-FJcycles}, we focus on the geometric side.

For test functions $f$, $f^\vee$ for $\pi, \pi^\vee$ respectively, and $\phi \in \Hom_E(A_K, A_\mu,\epsilon)$ and $\phi_c \in \Hom_E(A_K, A_\mu,-\epsilon)$. Consider the Beilinson–Bloch–Poincaré (BBP) height pairing (\cite[Conjecture 1.5]{Liu-FJcycles})
\[
\Int(f,\phi, f^\vee, \phi^\vee)=\vol(K)^2 \left< \mathrm{FJ}(f, \phi)_K, \mathrm{FJ}(f^\vee, \phi_c)_K \right>_{Y_K, A_\mu}^{\BBP}
\]
The BBP height pairing is defined via Beilinson–Bloch height pairings on $Y_K \times X_K \times X_K$ using the Poincare line bundle (a Cartier divisor) $\mathcal{P}_\mu \in \CH^1(A_\mu \times A_\mu^\vee)$. 

We may enlarge the field $E^\RSZ$ to a number field $E'$ such that $D_K^{n-1}$ is a sum of $E$-rational points of $X_K$. Use $(-)'$ to denote the base change to $E'$, for instance $X_K'=(X_K)_{E'} \to Y_K'=(Y_K)_{E'}$ and $A_\mu'=(A_\mu)_{E'}$. For a point $P \in D_K^{n-1}(E')$, we consider the Albanese morphisms $\alpha_P: X_K' \to A_\mu$.

We do a first reduction of the height pairing to Belinson-Bloch height pairing on $Y_K' \times X_K' \times X_K'$. For $P,Q \in X_K(E')$, let 
$$
\Delta^{\phi, P}X:=(\id_{Y_K'} \times (\phi' \otimes \alpha_P)_*)(\Delta^{E_0 \times F_0} X_K) \in \CH^{n-1+[M_\mu:\BQ]/2}(Y_K' \times A_\mu'),
$$
$$
\Delta^{\phi_c, Q}X:=(\id_{Y_K'} \times (\phi'_c \otimes \alpha_Q)_*)(\Delta^{E_0 \times F_0} X_K) \in \CH^{n-1+[M_\mu:\BQ]/2}(Y_K' \times A_\mu'^\vee).
$$
We modify these cycles to cohomological trivial cycles by replacing $\Delta^{E_0 \times F_0} X_K$ by $\Delta^{E_0 \times F_0}_z X_K=z^* \Delta^{E_0 \times F_0} X_K$ by taking a Hecke system of projectors $z=(z_K)_K: \CH^*(Y_K \times X_K) \to \CH^*(Y_K \times X_K)^0$
\cite[Definition 4.35]{Liu-FJcycles}. We always use the notation $\Delta_z$ to mean taking the Hecke projector $z$.

As \cite[Lemma 5.1.]{Liu-FJcycles}, above BBP height pairing $\Int(f,\phi, f^\vee, \phi^\vee)$ is a sum of similar height pairings
$$
\left <(T^f_K \otimes T_\mu^\can)^* \Delta^{\phi, P}_zX,  (T^{f^\vee}_K \otimes T_{\mu^c}^\can)^*   \Delta^{\phi_c, Q}_zX \right >_{Y_K', A_\mu}^{\BBP}.
$$
indexed by $P, Q \in D_K^{n-1}(E')$.

\begin{definition}
Let the \emph{weak magic divisor} for $(\phi, \phi_c)$ and $P, Q \in X_K(E')$ be the divisor 
\[
\mathcal{Q}^{\phi,\phi^c, P, Q}_{\mu, K}= (\phi' \circ \alpha_P) \times (\phi'_c \circ \alpha_Q)^* (T_\mu^{\can, t} \otimes T_\mu^{\can})^* \mathcal{P}_\mu \in \CH^1(X_K'\times X_K').
\]
\end{definition}

Consider $(y, x_1, x_2) \in Y_K' \times X_K' \times X_K'$. We have the divisor $Y_K' \times \mathcal{Q}^{\phi,\phi^c, P, Q}_{\mu, K} $ on $Y_K' \times X_K' \times X_K'$. Let $p_i: Y_K' \times X_K' \times X_K' \to Y_K' \times X_K'$ be the projection map sending $(y, x_1, x_2)$ to $(y,x_i)$. Consider the cycle 
$$
\Delta^{E_0 \times F_0}_{f, f^\vee} X_K :=(T^{f^t * f^\vee}_K \times id_{X_K'})^*  \Delta^{E_0 \times F_0} X_K \in \CH^{2(n-1)}(Y_K \times X_K).
$$

\begin{proposition}
For $P, Q \in X_K(E')$, we have
\[
\left <(T^f_K \otimes T_\mu^\can)^* \Delta^{\phi, P}_zX,  (T^{f^\vee}_K \otimes T_{\mu^c}^\can)^*   \Delta^{\phi_c, Q}_zX \right >_{Y_K', A_\mu}^{\BBP}
\]
\[
=(p_1^* \Delta_z^{E_0 \times F_0} X_K, (Y_K' \times \mathcal{Q}^{\phi,\phi^c, P, Q}_{\mu, K}). p_2^* \Delta^{E_0 \times F_0}_{z, f, f^\vee} X_K )^\BB_{Y_K' \times X_K' \times X_K'}.
\]
\end{proposition}
\begin{proof}
The proof follows from the argument of \cite[Lemma 5.2]{Liu-FJcycles}.
\end{proof}

Now we need to understand the three cycles 
$$
p_1^* \Delta^{E_0 \times F_0} X_K=\{(y,x_1,x_2)|y=x_1\} , \quad p_2^* \Delta^{E_0 \times F_0} X_K =\{(y,x_1,x_2)|y=x_2\} \in \CH^{2(n-1)}(Y_K' \times X_K' \times X_K')
$$
and the divisor $Y_K' \times \mathcal{Q}^{\phi,\phi^c, P, Q}_{\mu, K} \in \CH^{1}(Y_K' \times X_K' \times X_K')$. Weak magic divisors motivate the definition of doubling divisors \cite[Definition 5.8]{Liu-FJcycles} and is related to Kudla's generating series of divisors (see also Proposition \ref{modularity over generic fiber})
$$
Z(\phi)_K:=Z(h=1, \phi) \in \CH^1(X_K)
$$
on $X_K$ for $\phi \in \CS(V(\BA_f))^K$. 

Let $\mathrm{cl}_B: \CH^1(X_K' \times X_K') \to H^2_{B}(X_K \times X_K, \BC)$ by the cycle class map to Betti cohomology group of $(X_K \times X_K)(\BC)$. We have the following ''partial Fourier transform'' provided by Kudla-Millson theory \cite{Kudla-Millson}.

\begin{proposition} (\cite[Lemma 5.5, Lemma 5.6]{Liu-FJcycles})
There is a natural isomorphism of $\BC[\U(V)(\BA_f)]$-modules
\begin{equation}
\mathfrak{d}: \Omega(\mu,\epsilon)\otimes_{M_\mu} \Omega(\mu_c, - \epsilon) \otimes_{M_\mu} \BC \cong \CS(V(\BA_f)).
\end{equation}
Moreover, there is a unique $\BC$-linear map
\[
c\mathcal{Q}_{\mu, K}: \CS(V(\BA_f))^K \to H^2_{B}(X_K \times X_K, \BC)
\]
such that $c\mathcal{Q}_{\mu, K}(\mathfrak{\phi \otimes \phi_c})=\mathrm{cl}_B (\mathcal{Q}^{\phi,\phi^c, P, Q}_{\mu, K})$, and under the diagonal $\Delta: X_K \to X_K \times X_K$ we have
\[
\Delta^* c\mathcal{Q}_{\mu, K}(\phi)= \mathrm{cl}_B (Z(\phi)_K).
\]
\end{proposition}

\begin{definition}
A codimension $1$ cycle $Z^D_K(\phi)$ on $X_K \times X_K$ is called a doubling divisor (or magic divisor) of level $K$ for
an element $\phi \in \CS(V(\BA_f))^K$ if $\mathrm{cl}_B(Z_K) = c\mathcal{Q}_{\mu, K}(\phi)$, and $Z_K$ has proper intersection with the diagonal $\Delta: X_K \to X_K \times X_K$.
\end{definition}

Now consider the Shimura variety $Y_K \times Y_K \times X_K \times X_K$ and in which we denote a point by $(y_1, y_2, x_1, x_2) \in Y_K \times Y_K \times X_K \times X_K$. Choose a doubling divisor $Z_K^D(\phi)$ on $M \times M$ for $\phi$. We obtain the cycle 
$$
T^f_K \times Z_K \in \CH^{2(n-1)+1}(Y_K \times Y_K \times X_K \times X_K). 
$$
We have another two cycles
\[
P_{y_1,x_1}:=\{ y_1=x_1 \} \in \CH^{2(n-1)}(Y_K \times Y_K \times X_K \times X_K). 
\]
\[
P_{y_2,x_2}:=\{ y_2=x_2 \} \in \CH^{2(n-1)}(Y_K \times Y_K \times X_K \times X_K). 
\]
We denote by $P_{y_1,x_1}^z$ (resp. $P_{y_2,x_2}^z$) their modifications under the Hecke systems $z$ of projectors for $\CH^{*}(Y_K \times X_K)$.

\begin{definition}[Global arithmetic invariant functional]
$$
\mathcal{I}_K^{z}(f, \phi):= \left< P_{y_1,x_1}^z , (T^f_K \times Z^D_K(\phi)).P_{y_2,x_2}^z \right>^\BB_{Y_K \times Y_K \times X_K \times X_K} .
$$
\end{definition}

Note in the untwisted case, \cite[Definition 5.13]{Liu-FJcycles} uses a section $T_K \mapsto (\id \times T_K)$ from Hecke correspondences on $X_K$ to Hecke correspondences on $X_K \times X_K$, which is not available in our twisted set up. Similar to \cite[Proposition 5.10]{Liu-FJcycles}, the height pairing 
$$
(p_1^* \Delta_z^{E_0 \times F_0} X_K, (Y_K' \times \mathcal{Q}^{\phi,\phi^c, P, Q}_{\mu, K}). p_2^* \Delta^{E_0 \times F_0}_{z, f, f^\vee} X_K )^\BB_{Y_K' \times X_K' \times X_K'}
$$
is reduced to above global arithmetic invariant functionals $\mathcal{I}^z_K(-,-)$. To proceed, we consider the local part of the global arithmetic invariant functional at a good place $v_0$ of $F_0$.

\begin{definition}
We say a finite place $v_0$ of $F_0$ is a good place (with respect to $F$ and $\phi$), if 
\begin{enumerate}
    \item $v_0$ is inert in $F$ and unramified in $E_0$.
    \item the underlying rational prime $p$ is odd and unramified in $E^\RSZ$. 
    \item Let $v$ be any place of $F_0$ above $p$. Then there exists a self-dual lattice $L_v \subseteq V_v$ such that $K_v=\U(L_v), K_v^{E_0}=\U(L_{v, O_{E_0}})$. And $F_v=1_{K_v^{E_0}}$, $\phi_v=1_{L_v}$.
\end{enumerate}
\end{definition}

Fix a good place $v_0$ of $F_0$, and denote by $q_{v_0}$ the size of its residual field $k_{v_0}$. We assume that $v_0$ is inert in $E_0$ otherwise the reduction is done in \cite[Theorem 5.25]{Liu-FJcycles}. As Section \ref{subsection: integral model}, choose large enough $\Delta$ such that no $p$-adic places are in $\Delta$. Consider the embedding of punctured integral models over the semi-local ring $O_{E_{v_0}^\RSZ}$ 
\[
\Delta^{E_0}: \CX_K:=\CM_{O_{E_{v_0}}^\RSZ} \to \CY_{K}:=\CM^{E_0}_{O_{E_{v_0}}^\RSZ}
\]
which are proper smooth schemes. Let $\CP_{y_i,x_i}=\{y_i=x_i\}$ ($i=1,2$) in $\CY_K \times \CY_K \times \CX_K \times \CX_K$. 

\begin{definition}[Local arithmetic invariant functional at a good place $v_0$] Assume that $F \otimes \phi$ has regular supports at a place of $F_0$. The local arithmetic invariant functional at $v_0$ is the Euler–Poincaré characteristic
\[
\mathcal{I}_{v_0}(f, \phi):= (2 \log q_{v_0}) \chi( \CY_K \times \CY_K \times \CX_K \times \CX_K, \CO( \CP_{y_1,x_1}) \otimes^\BL \CO((T^f_K \times \CZ^D_K(\phi)).\CP_{y_2,x_2}) )
\]
\end{definition}

We expect the local arithmetic invariant functional is the main term of the global arithmetic invariant functional at $v_0$. Let $\CZ(\phi)_K$ be the Zariski closure of the geometric theta series $Z(\phi)_K$ on $\CX_K$ which is related to the arithmetic theta series of Kudla-Rapoport divisors in Section \ref{section: arithmetic theta}. 

\begin{proposition}
Assume the support of 
$$
\CP_{y_1,x_1} \cap (T^f_K \times \CZ^D_K(\phi) ) \cap \CP_{y_2,x_2} 
$$
is in $\CY_K \times \CY_K \times \Delta X_K$. Then we have 
\[
\mathcal{I}_{v_0}(f, \phi)= (2 \log q_{v_0}) \chi( \CY_K \times \CY_K,  \CO(\CZ(\phi)_K \times \CX_K )  \otimes^\BL \CO(T^f_K)   \otimes^\BL  \CO(\CX_K \times \CX_K) ).
\]
\end{proposition}
\begin{proof}
This follows from the same argument of \cite[Proposition 5.19]{Liu-FJcycles}.
\end{proof}

Note that the support assumption is automatic in the untwisted case \cite[Proposition 5.19]{Liu-FJcycles}. We expect the support condition to hold for good choices of $F$ and $Z_K^D(\phi)$. Then by the argument of \cite[Theorem 5.25]{Liu-FJcycles}, using Definition \ref{defn: global arithmetic intersection numbers} and Proposition \ref{prop: global int=local int}, we see that $\xi$-the Fourier coefficients of $\mathcal{I}_{v_0}(f, \phi)$  at a good place $v_0$ of $F_0$ is related to our intersection pairing (\ref{global int=sum of semi-global int at each v})
$$\Int_{v_0}(\xi, \Phi=f \otimes \phi),
$$ 
hence to the twisted arithmetic fundamental lemma (Theorem \ref{Twisted AFL theorem}) by Proposition \ref{prop: global int=local int}. See also Corollary \ref{cor: local AFL implies semi-global}.

\end{document}